\documentclass[preprint,aap]{imsart}
\usepackage[utf8]{inputenc}
\usepackage{lmodern}

\usepackage{a4wide}
\usepackage{amsfonts,amsmath,amssymb,amsthm}
\usepackage{bbm}
\usepackage{bm}
\usepackage[a4paper]{geometry}
\usepackage{hyperref}
\usepackage{multirow}
\usepackage{tikz} 
\usetikzlibrary{shapes.misc}
\usepackage{xcolor}
\usepackage{graphicx, subcaption}

\newcommand{\E}{\mathbb{E}}
\newcommand{\N}{\mathbb{N}}

\renewcommand{\P}{\mathbb{P}}

\newcommand{\R}{\mathbb{R}}
\let\dR\R
\let\dN\N

\newcommand{\cU}{\mathcal{U}}
\newcommand{\cV}{\mathcal{V}}
\newcommand{\cO}{\mathcal{O}}
\newcommand{\cW}{\mathcal{W}}
\newcommand{\cN}{\mathcal{N}}
\newcommand{\bu}{\mathbf{u}}
\newcommand{\bi}{\mathbf{i}}
\newcommand{\bt}{\mathbf{t}}
\newcommand{\bs}{\mathbf{s}}
\let\eps\varepsilon
\DeclareMathOperator{\argmin}{argmin}

\DeclareMathOperator{\Hess}{Hess}

\usepackage[normalem]{ulem}

\newcommand{\1}{\mathbf{1}}
\newcommand{\half}{\mbox{$\frac 1 2 $}}

\newcommand{\scal}[1]{\langle#1\rangle}
\newcommand{\abs}[1]{\left|#1\right|}
\newcommand{\PAR}[1]{\left(#1\right)}
\newcommand{\ind}[1]{\1_{#1}}
\newcommand{\cl}[1]{\mathrm{Cl}(#1)}
\newcommand{\fullyleadsto}{\looparrowright}
\newcommand{\prb}[2][]{\P_{#1}\left[#2\right]}
\newcommand{\esp}[2][]{\E_{#1}\left[#2\right]}

\newtheorem{corollary}{Corollary}
\newtheorem{lemma}{Lemma}
\newtheorem{proposition}{Proposition}
\newtheorem{assumption}{Growth Condition}
\newtheorem{theorem}{Theorem}

\theoremstyle{definition}
\newtheorem{definition}{Definition}

\theoremstyle{remark}
\newtheorem{example}{Example}
\newtheorem{remark}{Remark}

\begin{document}
	
	\begin{frontmatter}
		
		\title{Ergodicity of the zigzag process}
		\thankstext{T1}{Supported by the European Research Council under the European Union’s Seventh
			Framework Programme (FP/2007-2013) / ERC Grant Agreement number 614492, 
			the French National Research Agency under the grant ANR-12-JS01-0006 (PIECE)
			and the EPSRC grants EP/D002060/1 (CRiSM) and EP/K014463/1 (ilike). 
			}
		
		\begin{aug}
			\author{\fnms{Joris}  \snm{Bierkens}\corref{}\ead[label=e1]{joris.bierkens(AT)tudelft.nl}},
			\author{\fnms{Gareth O.} \snm{Roberts}\ead[label=e2]{g.o.roberts(AT)warwick.ac.uk}}
			\and
			\author{\fnms{Pierre-Andr\'e}  \snm{Zitt}%
				\ead[label=e3]{pierre-andre.zitt(AT)u-pem.fr}}%
			
			
			\runauthor{J. Bierkens, G. Roberts, P.-A. Zitt}
			
			\affiliation{Delft University of Technology, University of Warwick, Universit\'e-Paris-Est-Marne-La-Vall\'ee}
			
			\address{Joris Bierkens \\ Delft Institute of Applied Mathematics \\ Van Mourik Broekmanweg 6 \\ 2628 XE Delft \\ Netherlands. \\
				\printead{e1}}
			
			\address{Gareth O. Roberts \\ Department of Statistics \\ University of Warwick \\ Coventry CV4 7AL \\ United Kingdom. \\
	\printead{e2}}	
			
			\address{Pierre-Andr\'e Zitt \\ Laboratoire d'Analyse et Math\'ematiques Appliqu\'ees (UMR CNRS 8050) \\ Universit\'e-Paris-Est-Marne-La-Vall\'ee \\
					5, boulevard Descartes \\
					Cité Descartes, Champs-sur-Marne \\
					77454 Marne-la-Vallée Cedex 2, France. \\
				\printead{e3}}
			
		\end{aug}
		
		\begin{abstract}
  The zigzag process is a Piecewise Deterministic Markov Process 
which can be used in a MCMC framework to sample from a given target distribution. 
We prove the convergence of this process to its target under very weak 
assumptions, and establish a central limit theorem for empirical averages 
under stronger assumptions on the decay of the target measure. 
We use the classical ``Meyn-Tweedie'' approach \cite{MT2,MT09}. The main difficulty turns 
out to be the proof that the process can indeed reach all the points in the space, 
even if we consider the minimal switching rates.		\end{abstract}
		
		
		
	\end{frontmatter}


\section{Introduction}

\subsection{Motivation}

In recent years there has been a growing interest in the use of Piecewise
Deterministic Markov Process (PDMPs) within the
field of Markov Chain Monte Carlo (MCMC). In MCMC the objective is to simulate
from a `target' probability distribution $\pi$ by designing a Markov chain (or
process) which is ergodic and has stationary distribution $\pi$. Although in
principle MCMC, e.g. in the form of the Metropolis-Hastings algorithm
\cite{Metropolis1953}, can be used to sample from almost any probability
distribution of interest, it can suffer from slow convergence as well as heavy
computational cost per iteration. 

It is for exactly these two reasons that PDMPs are so promising. Firstly, PDMPs
are nonreversible, and it is known that nonreversible Markov processes may
offer faster convergence relative to reversible Markov processes (see e.g.
\cite{Bierkens2015,DiaconisHolmesNeal2000,DuncanLelievrePavliotis2015,Hwang1993,Lelievre2013,Ma2016,ReyBelletSpiliopoulos2015,TuritsynChertkovVucelja2011}) 
Secondly, a remarkable feature of the simulation procedure of 
some PDMPs  is that we can choose to use unbiased estimates of the `canonical'
switching rate without affecting the stationarity of $\pi$. In settings in
Bayesian statistics with large data sets (consisting of $n$ observations, say),
this offers significant benefits \cite{BierkensFearnheadRoberts2016}, reducing
computational effort per iteration from $\mathcal O(n)$ to $\mathcal O(1)$.
Similar computational benefits can be obtained in systems in statistical
physics consisting of many particles \cite{MichelKapferKrauth2014}. The use of PDMPs in sampling is a very active area of current research and (although it is not possible to give a complete list of references) we point the interested reader to \cite{BierkensDuncan2016,BouchardCoteVollmerDoucet2017,MichelKapferKrauth2014,Monmarche2016,Pakman2016,Pakman2017,PetersDeWith2012,Sherlock2017,Vanetti2017,Wu2017}.

The zigzag process (ZZP) is an example of such a Piecewise Deterministic Markov
Process. As the name suggests, PDMPs follow deterministic dynamics, inbetween 
random times where they may jump or change to another deterministic dynamics
(see \cite{Mal15,ABGKZ12} for examples and additional references). 
For example, in the ZZP in $\R^d$, trajectories $X_t$ have a piecewise
constant velocity $\Theta_t$ belonging to the finite set $\{-1,1\}^d$, with
components of the direction changing at random times
\cite{BierkensFearnheadRoberts2016}. These random times are generated from
inhomogeneous Poisson processes which have a space and direction dependent
switching rate $\lambda_i(X_t,\Theta_t)$, for switching the $i$-th component of
$\Theta_t$. Viewed as process in the state space $E := \R^d \times \{-1,1\}^d$,
$(X_t, \Theta_t)_{t \geq 0}$ is a Markov process. The switching intensities
$\lambda_i$ can be chosen in such a way that the marginal density on $\R^d$ of
the stationary probability distribution of $(X_t,\Theta_t)$ is equal to a prescribed
density function $\overline \pi$. Other variants of PDMPs with similar
properties exist, for example the Bouncy Particle Sampler (BPS,
\cite{BouchardCoteVollmerDoucet2017}) which selects its direction from $\R^d$
or the unit sphere in $\R^d$.

In order for a Markov process to be useful in MCMC, it should have the prescribed 
stationary distribution and furthermore the process should be \emph{ergodic}: 
the empirical time averages of a test function $f$ along a trajectory should 
converge to the space average $\int f d\pi$, a property that usually follows
from some kind of irreducibility, meaning roughly speaking that the process should be able to reach any point starting from any other point. 
The first requirement, stationarity, is relatively easy to
satisfy. However the second requirement is certainly non-trivial in the case of
PDMPs. For example, it is known that without `refreshments' of the velocity,
the BPS can be non-ergodic, for instance for any elliptically symmetric distribution such as a multivariate Gaussian \cite{BouchardCoteVollmerDoucet2017}. In contrast, it is
known that the ZZP is ergodic in certain cases in which the BPS is not ergodic
\cite{BierkensFearnheadRoberts2016}, and computer experiments have suggested
that in fact the ZZP is ergodic under only minimal assumptions. 
The main result of this paper is a proof of  ergodicity for the ZZP under very mild and reasonable conditions, giving theoretical justification for its use in MCMC. This gives the ZZP a possible
advantage over the BPS: the practitioner can be confident of the validity of
the ZZP as MCMC algorithm and does not need to worry about tuning a
refreshment parameter, which may slow down
convergence to equilibrium if chosen suboptimally. However other aspects are
also influential in determining speed of convergence and computational
efficiency, and the relative merits of the ZZP versus the BPS is an area of
challenging current and future research. See \cite{Andrieu2018,Bierkens2018a,Deligiannidis2018} for results in this direction.

Once ergodicity is established, one may look for estimates of rates of convergence
to the invariant measure, in various senses. One of the possible approaches to 
establish such results is to find a Lyapunov function. 
For nonreversible processes 
with small noise, it is often very difficult to guess the form 
of a suitable Lyapunov function, and quite technical to prove that it indeed works: 
see for example \cite{Durmus2018,Fet17,Deligiannidis2017}.
In the zigzag case, it turns out that under a reasonable assumption on the 
decay of the target measure $\pi$ at infinity, we are able to find a Lyapunov function
in a quite simple form. Leveraging well known results on long time convergence of 
processes, this proves in particular that the convergence towards the target 
measure $\pi$ occurs exponentially fast, and we also get a central limit theorem for ergodic averages. 

In \cite{BierkensRoberts2015} ergodicity of the one-dimensional zigzag process is established, which is significantly easier than the multi-dimensional case: for the one-dimensional process it is always possible to switch the single direction component along a trajectory, so that irreducibility is relatively straightforward. The examples of Section~\ref{sec:examples} illustrate why proving ergodicity in the multi-dimensional case is fundamentally different. The conditions for exponential ergodicity in the one-dimensional case are weaker than those we impose for the multi-dimensional case, which is due to the fact that the one-dimensional Lyapunov function does not carry over to the multi-dimensional case; see Section~\ref{sec:Lyapunov} for a brief discussion. From a practical viewpoint the slightly stronger conditions which we impose here are very reasonable.


\subsection{Preliminaries}
\label{sec:preliminaries}

We briefly recall the construction of the zigzag process in $E = \R^d \times \{-1,1\}^d$. For details we refer to \cite{BierkensFearnheadRoberts2016}. 

We equip $E$ with its natural product topology, so that a function $(x,\theta)
\mapsto f(x,\theta)$ is continuous if and only if $x \mapsto f(x,\theta)$ is continuous for
every $\theta$. Similarly $f$ is Lebesgue measurable if $x \mapsto
f(x,\theta)$ is measurable for every $\theta$.

For $i = 1, \dots, d$ introduce the mapping $F_i : \{-1,1\}^d \rightarrow
\{-1,1\}^d$ which flips the $i$-th component: For $j = 1, \dots, d$ and $\theta
\in \{-1,1\}^d$,
\[ (F_i \theta)_j = \begin{cases} \theta_j \quad & j \neq i, \\
                     - \theta_j \quad & j = i.
           \end{cases}
\]
Let $U:\dR^d\to \dR$ be a continuously differentiable \emph{potential function}. 
We introduce continuous \emph{switching intensities} (also referred to as
\emph{switching rates}) 
$\lambda_i : E \rightarrow [0,\infty)$, $i =1, \dots, d$, 
and assume that they are linked with the potential through the relation
\begin{equation}
  \label{eq:relation_lambda} \lambda_i(x,\theta) - \lambda_i(x,F_i\theta)  = \theta_i \partial_i U(x), \quad  (x,\theta) \in E,  i = 1, \dots, d.
\end{equation}

An equivalent condition on the switching rates is the existence of a continuous function $\gamma : E \rightarrow [0,\infty)^d$ whose $i$-th component does not depend on $\theta_i$,
\begin{equation} \label{eq:condition-gamma}
  \gamma_i(x, F_i\theta) = \gamma_i(x, \theta), \quad (x,\theta) \in E, i =1,\dots,d,\end{equation}
and which is related to the switching rate through
\begin{equation} \label{eq:lambda-explicit} 
  \lambda_i(x,\theta) = (\theta_i \partial_i U(x))_+ + \gamma_i(x,\theta), \quad (x,\theta) \in E, i =1, \dots, d.
\end{equation}
Here $(a)_+ := \max(0,a)$ is the positive part of $a \in \R$.
We call $\gamma$ the \emph{excess switching intensity} and $\lambda$ satisfying~\eqref{eq:lambda-explicit} with  $\gamma \equiv 0$ the \emph{canonical switching intensity}.

For $(x,\theta) \in E$, we construct a trajectory of $(X,\Theta)$ of the zigzag process with initial condition $(x, \theta)$ and switching intensities $\lambda(x,\theta)$ as follows. First we construct a finite or infinite sequence of \emph{skeleton points} $(T^k, X^k, \Theta^k)$ in $\R_+ \times E$ by the following iterative procedure.

\begin{itemize}
 \item Let $(T^0, X^0, \Theta^0) := (0, x,\theta)$.
 \item For $k = 1, 2, \dots$ 
 \begin{itemize}
 \item Let $x^k(t) := X^{k-1} + \Theta^{k-1} t$, $t \geq 0$
 \item For $i = 1, \dots, d$, let $\tau^k_i$ be distributed according to
 \[ \P(\tau^k_i \geq t) = \exp \left( - \int_0^t \lambda_i(x^k(s), \Theta^{k-1}) \ d s \right).\]
 \item Let $i_0 := \argmin_{i \in \{1, \dots, d \}} \tau^k_i$ and let $T^k := T^{k-1}+\tau^k_{i_0}$. In principle, it is possible that $\tau^k_i = \infty$ for all $i$ in which case the value of $i_0$ will turn out to be irrelevant and we set $T^k := \infty$.
 \item If $T^k < \infty$ let $X^k :=x^k(T^k)$ and $\Theta^k = F_{i_0} \Theta^{k-1}$ and repeat the steps. If $T^k = \infty$, terminate the procedure.
\end{itemize}
\end{itemize}
The piecewise deterministic trajectories $(X_t, \Theta_t)$ are now obtained as 
\[ (X_t,\Theta_t) := (X^k + \Theta^k (t- T^k), \Theta^k), \quad \quad  \mbox{$t \in [T^k, T^{k+1})$}, \quad k = 0, 1, 2, \dots,\]
defining a process in $E$ with the strong Markov property.

Informally, the process moves in straight lines, only changing velocities at the times $T^k$. 
  In the case of canonical switching rates $\lambda_i(x,\theta) = (\theta_i\partial_i U(x))_+$, 
  a change in the $i$\textsuperscript{th} component $\theta_i$ of the 
  velocity may only happen when in this direction, the process is going ``uphill'', that is, if $ \theta_i \partial_i U(x) >0$. 
  Note in particular that if following the current velocity increases $U$, 
  then $\scal{\theta, \nabla U(x)}>0$ and at least one of the components 
  has a positive rate of jump. 

We further impose an integrability condition on the potential function:
\begin{equation} \label{eq:integrability-potential} Z:= \int_{\R^d} \exp(-U(x)) \, d x < \infty.\end{equation}
Under this condition the zigzag process has a stationary probability distribution given by
\[ \pi(A \times \{ \theta\}) = \frac{1}{2^d Z} \int_A \exp(-U(x)) \, d x, \quad \quad \mbox{$A$ Lebesgue measurable and $\theta \in \{-1,1\}^d$}.\]
We will use the notation $\overline \pi(\cdot)$ for the marginal density function on $\R^d$, i.e. $\overline \pi(x) = \exp(-U(x))/Z$, $x \in \R^d$.

\subsection{Why ergodicity of the ZZP is non-trivial}
\label{sec:examples}

First consider a simple non-problematic case, where at every point in space all
switching rates $\lambda_i$ are positive. This can be achieved by letting
$\lambda_i(x,\theta) = \max(0, \theta_i \partial_i U(x)) + \gamma(x)$ where the
excess switching rate $\gamma : \R^d \rightarrow (0,\infty)$ assumes only
positive values. At an intuitive level, it is reasonable that such a process
can reach any point in the state space, since by making a certain number of
switches we can change direction to any direction in $\{-1,1\}^d$. These
directions span $\R^d$. After reaching an arbitrary point in $\R^d$ we can
switch to any desired final direction. Although we can not change direction
instantaneously but only over a time interval of positive length, the method
above enables us to reach any point in $\R^d \times \{-1,1\}^d$ to arbitrary
precision (and in fact, as will turn out, exactly).

However, having non-zero values for $\gamma(x,\theta)$ is not beneficial for
efficiency: the zigzag process becomes more diffusive as $\gamma_i$ increases
which results in higher computational costs, see e.g. \cite{BierkensDuncan2016}
for a detailed investigation of this phenomenon in the one-dimensional case.
Therefore we are mainly interested in the question of ergodicity for the case
in which $\gamma_i(x,\theta) = 0$ for all $i$, $x$ and $\theta$, i.e. for the
canonical switching rates.

The expression for the canonical switching rates immediately tells us that one
or more of the components of $\lambda$ are zero in large parts of the state
space. If the switching rate is zero on a set, it means that while the
trajectory moves within this set, there is no freedom to switch the components
of the direction vector. As a consequence it is far from obvious how to
construct trajectories between any two given points $(x,\theta)$ and $(y,\eta)$
in the state space, which could be a realization of a canonical ZZP trajectory.

To illustrate the difficulties, let us discuss three examples highlighting
what could go wrong with the zigzag process.

\begin{example}[A non-smooth example]

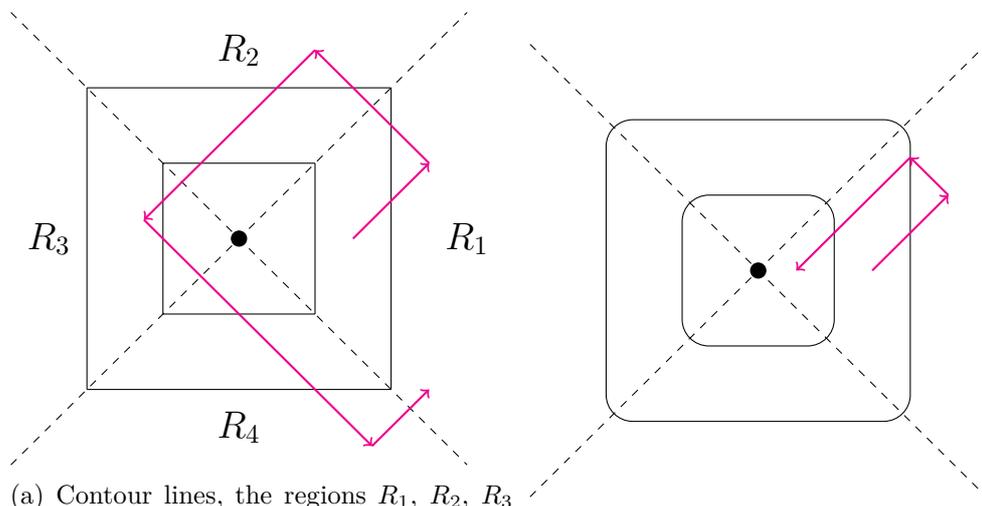
\begin{figure}[ht!]
{\centering 
\begin{subfigure}{0.45 \textwidth}\begin{tikzpicture}

\draw [fill = black]  (0,0) circle (0.1);

// inner contour
\draw (1,-1) -- (1,1);
\draw (1,1) -- (-1,1);
\draw (-1,1) -- (-1,-1);
\draw (-1,-1) -- (1,-1);

// outer contour
\draw (2,-2) -- (2,2);
\draw (2,2) -- (-2,2);
\draw (-2,2) -- (-2,-2);
\draw (-2,-2) -- (2,-2);

\draw [dashed] (-3,-3) -- (3,3);
\draw [dashed] (-3, 3) -- (3,-3);

\node at (3.0, 0.0) {\Large $R_1$};

\node at (0.0, 2.5) {\Large $R_2$};

\node at (-2.5, 0.0) {\Large $R_3$};

\node at (0.0, -2.5) {\Large $R_4$};

\draw [->, thick, magenta] (1.5,0) -- (2.5, 1);

\draw [->, thick, magenta] (2.5, 1) -- (1, 2.5);

\draw [->, thick, magenta] (1, 2.5) -- (-1.25, 0.25);

\draw [->, thick, magenta] (-1.25, 0.25)-- (7/4, -11/4);

\draw [->, thick, magenta] (7/4, -11/4) -- (2.5, -2);

\end{tikzpicture}\caption{Contour lines, the regions $R_1$, $R_2$, $R_3$ and $R_4$, and a typical trajectory for the potential function $U(x) = \max(|x_1|,|x_2|)$. From the displayed starting position it is impossible to reach a point in $R_1$ with direction $(-1,-1)$.}\end{subfigure}
\ \begin{subfigure}{0.45 \textwidth}\begin{tikzpicture}

\draw [fill = black]  (0,0) circle (0.1);

// inner contour
\draw [rounded corners=10pt]  (-1,-1) rectangle (1, 1) ;

// outer contour
\draw [rounded corners=10pt] (-2,-2) rectangle (2,2);

\draw [dashed] (-3,-3) -- (3,3);
\draw [dashed] (-3, 3) -- (3,-3);





\draw [->, thick, magenta] (1.5,0) -- (2.5, 1);

\draw [->, thick, magenta] (2.5, 1) -- (2, 1.5);

\draw [->, thick, magenta] (2, 1.5) -- (.5,0);





\end{tikzpicture}\caption{Once we `smoothen' the potential function slightly, it becomes possible to switch the second coordinate of the direction vector, making the process irreducible.} \end{subfigure}
\caption{The canonical zigzag process for $U(x) = \max(|x_1|,|x_2|)$ and a smoothed version of $U$.}
\label{fig:maximum}
}
\end{figure}

As an example of what can go wrong, consider the potential function $U : \R^2 \rightarrow \R$ given by $U(x) = \max(|x_1|,|x_2|)$. Having only a weak derivative, this example falls just outside the assumptions we will make in the formulation of the main results.
Ignoring the diagonals $x_2 = x_1$ and $x_2 = -x_1$, divide the plane into four regions:
\begin{align*} R_1 & = \{ (x_1,x_2) : x_1 > |x_2|\}, \quad & R_2 & = \{ (x_1,x_2) : x_2 > |x_1|\}, \\
R_3 & = \{ (x_1,x_2): x_1 < -|x_2|\}, \quad & R_4 & = \{(x_1,x_2) : x_2 < -|x_1|\}.
\end{align*}
The potential $U$ is almost everywhere differentiable, with
\[ \partial_1 U(x_1,x_2) = \begin{cases} 1 & \text{in } R_1,  \\
			-1 & \text{in } R_3, \\
                        0 & \text{in } R_2 \cup R_4,
                       \end{cases} \quad \mbox{and} \quad  \partial_2 U(x_1,x_2) = \begin{cases} 1 & \text{in } R_2, \\
			-1 & \text{in } R_4, \\
			 0 & \text{in } R_1 \cup R_3,
                       \end{cases}\]
and except for pathological initial values (along the diagonals), the switching
rates are well defined (albeit discontinuous) and we can construct a zigzag
process with these switching rates.  Suppose we start a trajectory with initial
condition $(x_1,x_2) \in R_1$ and initial direction $\theta = (+1,+1)$. The
trajectory will remain in $R_1$ at least until one of the components is
switched. The only component which has a positive switching rate is the first
component: $\lambda_1(x,\theta) = 1$ and $\lambda_2(x,\theta) = 0$ for $x \in
R_1$ and $\theta = (+1,+1)$. Therefore we will switch at some point to the
direction $(-1,+1)$, after which we will eventually reach the region $R_2$. We
can repeat this argument to find that, with full probability, we will
subsequently enter the regions $R_3$, $R_4$ and $R_1$ with directions
$(-1,-1)$, $(+1,-1)$ and $(+1,+1)$, respectively. In particular, from the given
initial condition it is impossible to reach a point in $R_1$ with a direction
$\theta$ for which $\theta_2 = -1$, and we conclude that the zigzag process is
not irreducible. If we consider a slightly smoothed version of the potential
function the associated zigzag process is irreducible on the combined
position-momentum space $E = \R^2 \times \{-1,1\}^2$. See
Figure~\ref{fig:maximum} for an illustration of this example.
\end{example}

\begin{example}[Gaussian distributions]

\label{ex:gaussian}

In this example we consider what may go wrong in the fundamental case of a Gaussian target distribution. Consider first the standard normal case, $U(x) = \half \|x\|^2$, so that $\nabla U(x) = x$ and $\lambda_i(x,\theta) = \max(0, \theta_i x_i)$. As a result, starting from $(x,\theta)$,
\[
  \lambda_i(x + \theta t, \theta) = (\theta_i (x_i + \theta_i t))_+ = ( \theta_i x_i + t)_+.
\]
We see that in this situation, as $t$ increases, eventually the switching rate
in any component becomes positive. This means that after travelling in a
certain direction, we may switch any component of the direction vector. The
same holds for Gaussian distributions with a diagonally dominant inverse
covariance matrix. In our first attempts to prove irreducibility this provided
us with a concrete way of building trajectories between any two points.

\begin{figure}
\begin{center}
\begin{subfigure}[b]{0.4 \textwidth}{\includegraphics[width = \textwidth]{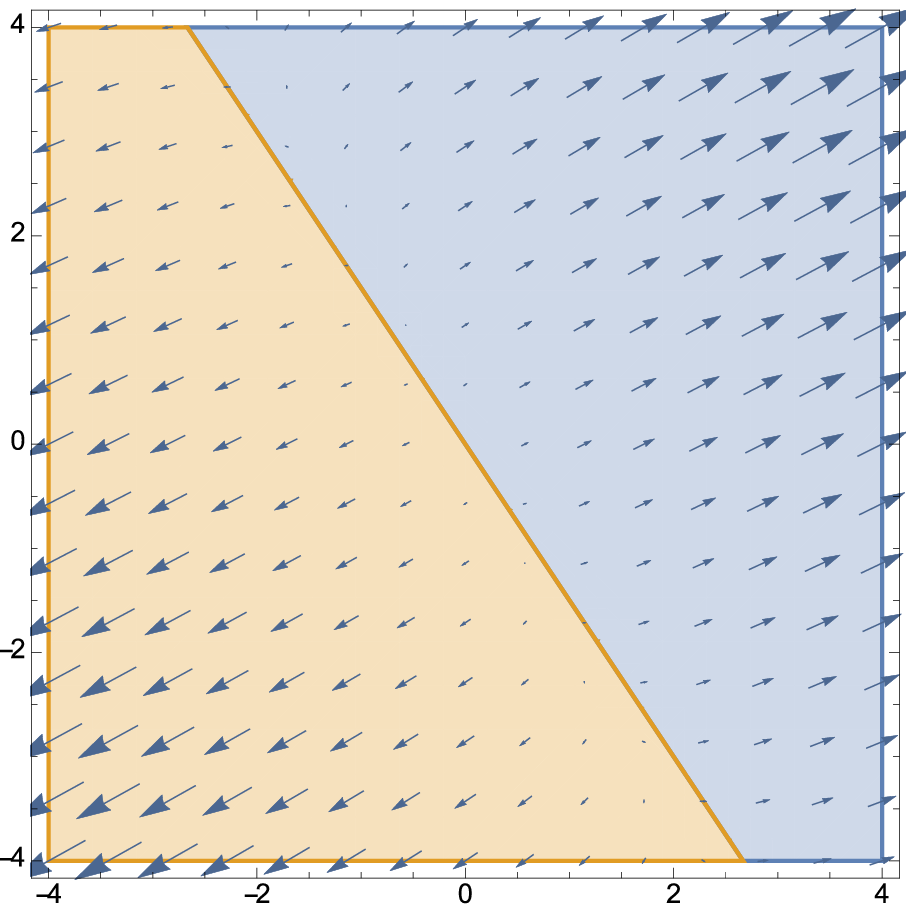}} 
\caption{The gradient vector field $\nabla U$}
\end{subfigure}
\begin{subfigure}[b]{0.4 \textwidth}{\includegraphics[width = \textwidth]{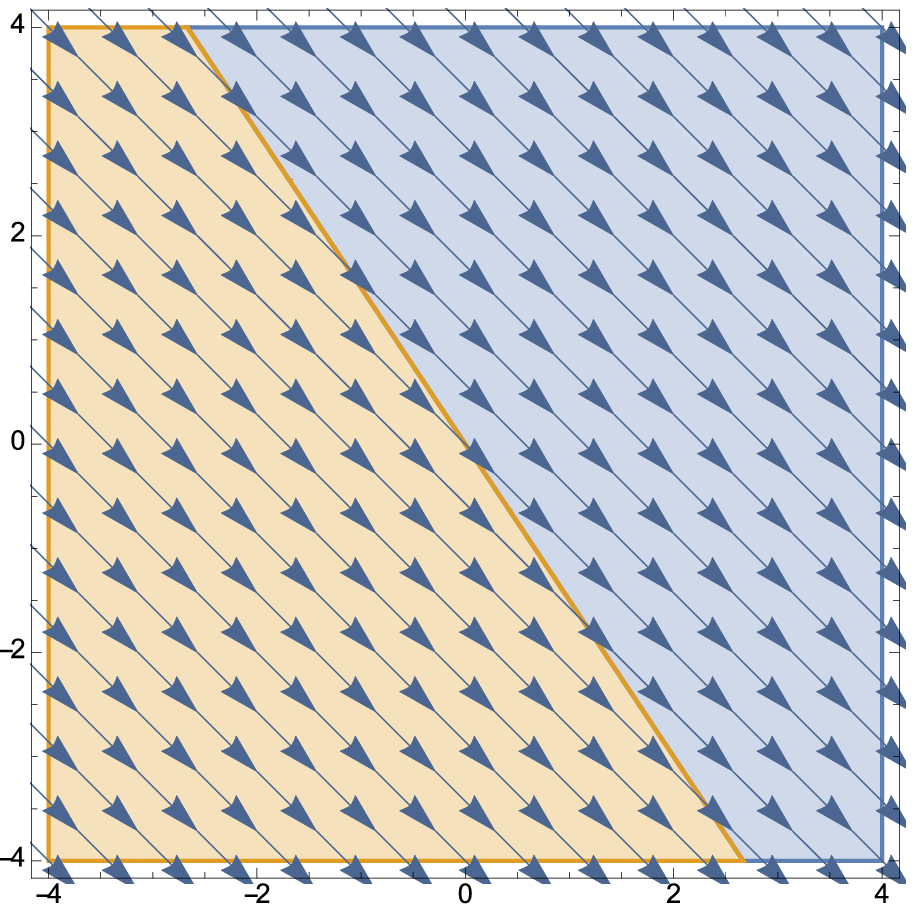}} 
\caption{The constant vector field $(+1,-1)$}
\end{subfigure}
\end{center}

{\small
An example in the setting of Example~\ref{ex:gaussian} in which the switching
rate in the second coordinate drops to zero after being non-zero initially.
Consider a two-dimensional Gaussian target distribution, with potential
function $U(x) = \half x^{\top} V x$, where $V = \begin{pmatrix}  6 & 3 ; &  3
  & 2 \end{pmatrix}$ (which is positive definite, but not diagonally dominant).
In Figure (a) the gradient field of $U$ is drawn. The region where $\partial_2
U > 0$ is shaded blue. In Figure (b) the constant vector field $\theta =
(+1,-1)$ is superimposed over the division between regions. If a trajectory
follows this vectorfield, coming from the yellow region where $\partial_2 U <
0$, at some point it enters the blue region. At this point the switching rate
for $\theta_2$, i.e. $\lambda_2(x,\theta) = \max(0, -\partial_2 U(x))$, drops
to zero. The conclusion is that switching rates of individual components are
not necessarily strictly increasing along the piecewise linear segments of the
trajectory, contrary to what intuition may suggest.
}
\caption[Example~\ref{ex:gaussian}]{A non diagonally dominant Gaussian case}
\label{fig:switch_rate_drops_to_zero}
\end{figure}

However, we should be careful since it is not always the case that, for large enough $t$, we can switch any component of the direction vector, even in ideal situations (e.g. with a strictly convex potential).
For example in a two dimensional Gaussian case, it may happen that the switching rate in a certain component may drop from being positive to zero as time increases. See Figure~\ref{fig:switch_rate_drops_to_zero} for an illustration of this phenomenon.

\end{example}

\begin{example}[Ridge]
\label{ex:ridge}

\begin{figure}
  \begin{minipage}{0.5\textwidth}
\includegraphics[width=\textwidth]{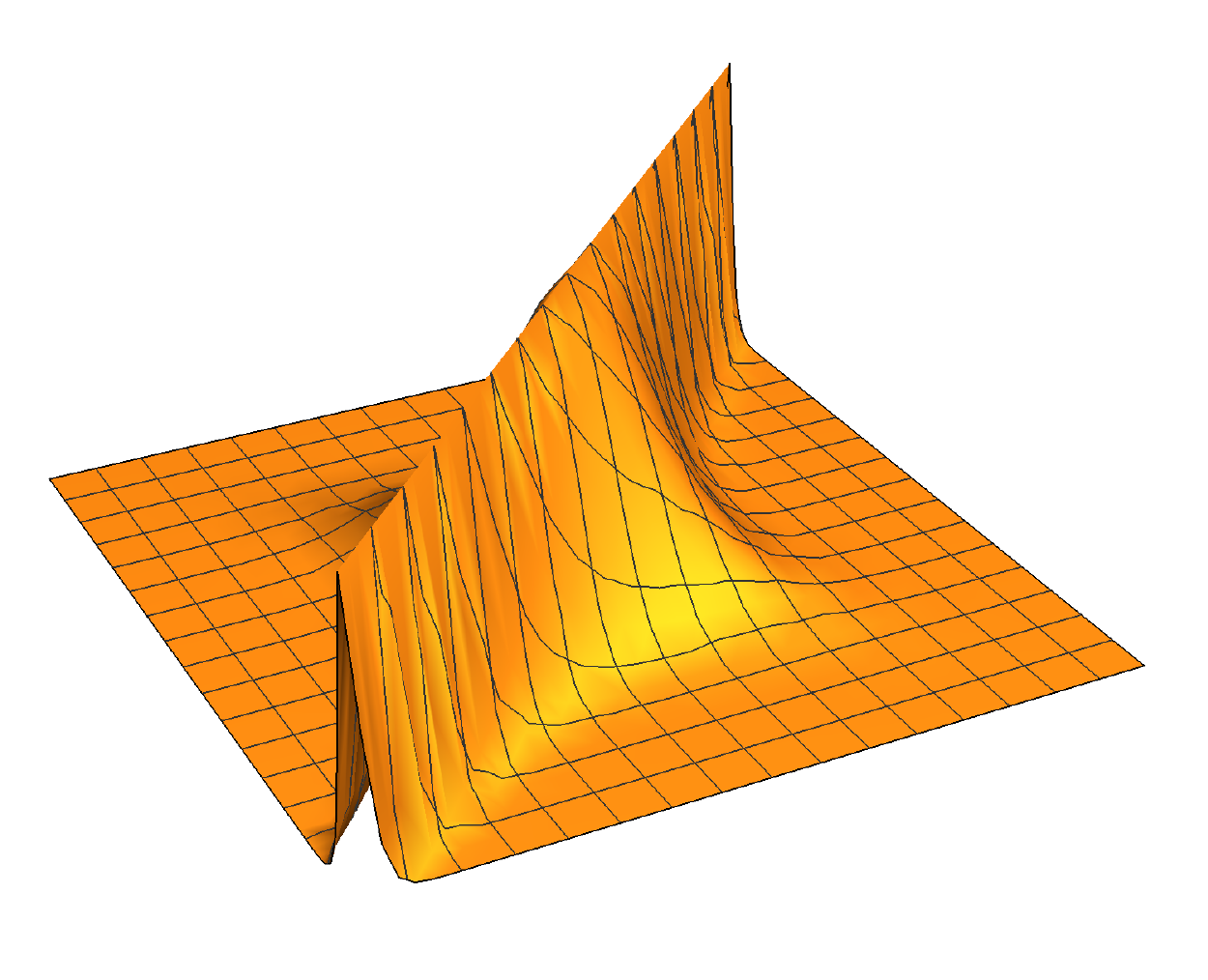}
\end{minipage}
\hfill
\begin{minipage}{0.4\textwidth}
  \small
A continuously differentiable probability density function in two dimensions which has the property that along a narrow ridge the slope vanishes.
\end{minipage}
\caption{The ``ridge'' example}
\label{fig:ridge}
\end{figure}

Consider a two-dimensional case in which $U(x_1,x_2) = |x_1-x_2|^{2 \alpha} (1
+ |x_1+x_2|^2)$, where $\tfrac 1 2 < \alpha < 1$. Note that $U(x_1,x_2)$ is
continuously differentiable and it can be seen that $\int_{\R} \int_{\R}
\exp(-U(x_1,x_2)) \, d x_1 \, d x_2< \infty$, so that $U$ is (after
normalization) the potential of a probability distribution on $\R^2$.  However
a simple computation yields that the gradient $\nabla U$ vanishes along the
diagonal $x_2 = x_1$, which is oriented with the directions $\pm (1,1)$. As a
consequence, starting from some initial condition $(x_1,x_2)$ satisfying $x_2 =
x_1$ in the direction $\pm (1,1)$, it will be impossible to switch any
component of the direction vector and inevitably we will drift off to infinity.
The function $\exp(-U(x_1,x_2))$ corresponds to a narrow ridge, along which the
derivative of $U$ vanishes; see Figure~\ref{fig:ridge}. As we will see, it is
essentially the fact that $U(x_1,x_2) \not\to \infty$ as $(x_1,x_2) \rightarrow
\infty$ which results in this evanescent behaviour. The lack of a nondegenerate
local minimum (our other fundamental assumption to prove irreducibility) is
less problematic. This is because the shape of $U$ can be modified smoothly
around the origin to have a local nondegenerate minimum, without removing the
possibility of drifting away to infinity.

\end{example}

\subsection{Main results}
\label{sec:mainresults}

We introduce three `growth conditions', i.e. conditions on the tail behaviour of the potential function.

\begin{assumption}
\label{ass:GC1}
$U \in \mathcal C^2$ and $\lim_{|x| \rightarrow \infty} U(x) = \infty$.
\end{assumption}

\begin{assumption}
\label{ass:GC2}
$U \in \mathcal C^2$ and for some constants $c > d$, $c' \in \R$, $U(x) \geq c \ln (|x|) - c'$ for all $x \in \R^d$.
\end{assumption}

\begin{assumption}
\label{ass:GC3}
$U \in \mathcal C^2$, 
\[  \lim_{|x| \rightarrow \infty} \frac{\max(1, \| \Hess U(x)\|)}{\abs{\nabla U(x)}} =0, \quad \text{and} \quad 
 \lim_{|x| \rightarrow \infty} \frac{\abs{\nabla U(x)}}{U(x)} = 0.\]
\end{assumption}

The following theorems are the main results of this paper.

\begin{theorem}[Ergodicity]
\label{thm:ergodicity}
Suppose the potential function is $\mathcal C^3$, has a nondegenerate local
minimum and satisfies Growth Condition~\ref{ass:GC2}. Then the zigzag process
is \emph{ergodic}, in the sense that 
\[ 
  \lim_{t \rightarrow \infty} \|\prb[(x,\theta)]{ (X_t,\Theta_t) \in \cdot} - \pi \|_{\mathrm{TV}} = 0
  \quad \text{for all $(x,\theta) \in E$.}
\]
\end{theorem}

The proof of Theorem~\ref{thm:ergodicity}  also establishes that the process 
is positively Harris recurrent (see Section~\ref{sec:probabilistic} below for a precise definition), so 
that the  Law of Large Numbers holds (see e.g. \cite{Maruyama1959}): for all initial conditions
$(x,\theta) \in E$ and $g \in L^1(\pi)$ for which $s \mapsto g(X_s, \Theta_s)$ is almost surely locally integrable,
\[ 
\lim_{T \rightarrow \infty} \frac 1 T \int_0^T g(X_s,\Theta_s) \, ds = \pi(g), \quad \text{almost surely}.
\]

\begin{theorem}[Exponential ergodicity]
\label{thm:exponential-ergodicity}
Suppose $U \in \mathcal C^3$, $U$ has a nondegenerate local minimum and Growth Condition~\ref{ass:GC3} is satisfied. Suppose the excess switching rates $(\gamma_i)_{i=1}^d$ are bounded. Then the zigzag process is \emph{exponentially ergodic}, that is, 
there exists a function $M:E\to \dR_+$ and a constant $c>0$ such that
 \[
   \| \prb[(x,\theta)]{ (X_t,\Theta_t) \in \cdot} - \pi \|_{\mathrm{TV}}
   \leq M(x,\theta)e^{-ct} \quad \text{for all $(x,\theta) \in E$ and $t \geq 0$.}
 \]
\end{theorem}

In particular, the Theorem~\ref{thm:exponential-ergodicity} allows for the case of canonical switching rates, i.e. $\gamma \equiv 0$.

\begin{remark}
Many target distributions which do not satisfy GC\ref{ass:GC3} can be transformed by a suitable change of variables after which GC\ref{ass:GC3} will be satisfied and exponential ergodicity can be obtained for the transformed distribution. The trajectories of the transformed process can then be used to compute ergodic averages approximating the intended target distribution. We refer to \cite{Deligiannidis2017, johnson2012variable} for details of this approach.
\end{remark}

\begin{remark}
Theorem~\ref{thm:exponential-ergodicity} establishes exponential ergodicity
under reasonable conditions (i.e. comparable to other sufficient conditions for
establishing exponential ergodicity of other processes
\cite{Deligiannidis2017,roberts1996quantitative,Stramer1999}) on the tails of the target
distribution. E.g. for potential functions of the form $U(x) = (1 + \|x\|^2)^{\alpha/2}$, Theorem~\ref{thm:exponential-ergodicity} establishes exponential ergodicity for any $\alpha > 1$.
For heavier tails, it is not yet clear what would be a suitable
Lyapunov function and this remains a topic of current research.
\end{remark}

\begin{remark}
Although GC\ref{ass:GC3} does not seem to imply GC\ref{ass:GC2}, it does imply non-evanescence through a Lyapunov argument \cite[Theorem 3.1]{MT3}.
\end{remark}

Under essentially the same conditions, we can also establish a Functional Central Limit Theorem. 
In the following theorem, we write $D[0,1]$ for the Skorohod space of cadlag functions on $[0,1]$.

\begin{theorem}[Functional Central Limit Theorem]
\label{thm:FCLT}
Suppose that $U \in \mathcal C^3$, $U$ has a nondegenerate local minimum, 
Growth Condition~\ref{ass:GC3} is satisfied, and $U$ satisfies the 
integrability condition  $\int_{\R^d} \exp(-\eta U(x)) \, d x < \infty$
for some $0 < \eta<1$.  Suppose the excess switching rates
$(\gamma_i)_{i=1}^d$ are bounded.

Let  $g : E \rightarrow \R$  satisfy $\abs{g(\cdot)} \leq k \exp(\beta U(\cdot))$ on $E$ for some $k > 0$ and $0\leq \beta < (1-\eta)/2$.   

Define $Z_n(t) := \frac 1 {\sqrt{n}} \int_0^{nt}
(g(X_s,\Theta_s) - \pi(g)) \, ds$, $t \geq 0$. 

There exists a $0 \leq \sigma_g < \infty$ such that
for any starting distribution,  
$Z_n$ converges in distribution in $D[0,1]$ to $\sigma_g B$, where 
$B$ is a standard brownian motion.
\end{theorem}

In particular, under the conditions of Theorem~\ref{thm:FCLT} the Central Limit
Theorem of ergodic averages holds:
\[ 
  \frac 1 {\sqrt{T}}\int_0^T (g(X_s,\Theta_s) - \pi(g)) \, ds 
  \stackrel{d}{\rightarrow} 
  N(0,\sigma_g^2) \quad \text{as $T \rightarrow \infty$.}
\]
\begin{remark} If $U$ grows faster than a positive power of $\abs{x}$, 
 then the integrability condition will be satisfied for $\eta$ arbitrarily small, 
and the CLT applies as soon as $\abs{g(\cdot)} \leq k \exp(\beta U)$ for 
some $\beta < 1/2$. In other words it applies for ``almost'' all functions 
$g\in L^2(\pi)$. 
\end{remark}

\begin{remark}
A CLT for the one-dimensional Zig-Zag process was obtained earlier in \cite{BierkensDuncan2016}.
\end{remark}

\subsection{Strategy}

\begin{figure}
\begin{tikzpicture}[align=center,node distance=10mm,%
  common/.style = {
    rounded rectangle, minimum size=10mm,
    thick,draw=black!70}, 
  condition/.style={
    common, top color=white,bottom color=black!20}, 
  deterministic/.style={
    common, top color=white, bottom color=red!20},
  probabilistic/.style={
    common, top color=white, bottom color = blue!20},
  ref/.style={
    midway,rectangle,fill=white,fill opacity=1,thick,draw=black!80}]
  \matrix[row sep=14mm,column sep=8mm]{
    & \node[condition] (uinf) { GC1: $U\to \infty$}; &[2em] &
    \node[condition] (ulog)  {GC2: $U\geq c\ln$} ; &
    &
    \\
    \node[condition] (locmin) {Local min.}; &  
    \node[deterministic] (eff) {Full flippability}; &
    \\
    &  \node[deterministic] (reach) {Reachability};  & & &
    \\
    \node[condition] (lyap)  {GC3} ;
    &\node[probabilistic] (irr) {$\psi$-irreducibility,\\ aperiodicity}; & 
    \node[probabilistic] (tproc) {$T$-process}; &
    \node[probabilistic] (sne) {Non evanescence}; &
    \\
    \node[probabilistic] (erg) {Exp. ergodicity};
    & &  \node[probabilistic] (harris) {Positive Harris\\ recurrence; ergodicity}; & & \\
   };
   \path[thick,->] (ulog) edge node[ref] {Prop. \ref{prop:SNE}} (sne) 
         (uinf) edge node[ref] {Prop. \ref{prop:full_flip}} (eff)
	 (eff) edge  (reach) 
	       edge[to path={-|(\tikztotarget) \tikztonodes}] node[ref] {Th. \ref{thm:Tprocess}} (tproc) 
	       (reach) edge node[ref] {Th. \ref{thm:Tprocess}} (irr)
	 (tproc) edge (harris)
	 (lyap) edge (erg);
   \draw[thick,->] (lyap) -- (erg);
   \draw[thick] (irr.290) -- ++(0,-0.5) node (tmp) {} -| (sne)
   (locmin.south) -- ++(0,-0.5) -| node[ref] {Th. \ref{thm:reachability}} (eff)
		(irr.250) -- ++(0,-0.5) -| node [ref] {Th.\ref{thm:exponential-ergodicity}} (lyap);
   \path (tmp) -| node[ref] {Th. \ref{thm:Tprocess}} (harris);
   
\end{tikzpicture}

{\small 
  Schematic overview of key properties of the zigzag process in relation to the
  Growth Conditions~\ref{ass:GC1},~\ref{ass:GC2} and~\ref{ass:GC3}. The grey
  nodes represent conditions on the potential $U$, the red nodes refer to
  deterministic `reachability' properties of trajectories, discussed in
  Section~\ref{sec:reachability}, and the blue nodes represent probabilistic
  properties discussed in Section~\ref{sec:probabilistic}.}
\caption{The key properties}
\label{fig:diagram}
\end{figure}

The diagram in Figure~\ref{fig:diagram} illustrates how the different Growth Conditions of Section~\ref{sec:mainresults} are related to key properties of the zigzag process, which are crucial to establish the main results.
As seen in the diagram, it is possible to distinguish between `deterministic' results and `probabilistic' results. 

The `deterministic' results, discussed in Section~\ref{sec:reachability} concern the control theoretic aspects of zigzag trajectories. Here we are concerned with \emph{reachability}: the existence of zigzag trajectories between any points in the state space such that, for a given potential function $U$, the trajectories are admissible: the switching intensities should be positive at the times at which the trajectory changes direction, even in the case of canonical switching rates. As a weaker notion, we are also interested in \emph{full flippability}: can we, starting from any point in the state space, be certain that eventually all components of the direction vectors are switched at least once? This will all be made more precise in Section~\ref{sec:reachability}.

Next, in the `probabilistic' section, Section~\ref{sec:probabilistic},  the
results of Section~\ref{sec:reachability} are employed in order to establish
several key properties ($\psi$-irreducibility, aperiodicity, the $T$-process
property, non-evanescence and (positive) Harris recurrence) of the zigzag
process as a Markov process, which finally result in proofs of the main
theorems. The definitions of these probabilistic notions, which are standard in
the Markov process literature \cite{MT2, MT09}, are recalled in the
introduction of Section~\ref{sec:probabilistic}.
We conclude with proofs of the main results, located in Section~\ref{sec:mainresults-proofs}.

\section{Reachability}
\label{sec:reachability}

\subsection{Admissible control sequences}
We define a \emph{control sequence} to
be a tuple $\mathbf u = (\mathbf t, \mathbf i)$, where $\mathbf t=(t_0,\dots,t_m) \in (0,\infty)^{m +
  1}$ and $\mathbf i = (i_1,\dots,i_m) \in \{ 1, \dots, n\}^m$ for some $m \in \N$. 
Starting from $(x,\theta)$ at time $0$, this sequence gives rise to 
a trajectory $(x(t),\theta(t))$ by: 
following $\theta$ for a time $t_0$, switching the $i_1$\textsuperscript{th}
component of $\theta$, following the new velocity for a time $t_1$, etc. 

More formally, writing $\tau_k = \sum_{i=0}^{k-1} t_i$ with the usual convention $\tau_0=0$, we define $(x(t),\theta(t))$ on $[0,\tau_{m+1}]$ by
 \begin{align*}
   \theta(t) &= F_{(i_1,\dots,i_k)} \theta, 
   \quad  \text{when  } \tau_k  \leq t < \tau_{k+1} \text{ for } k = 0, \dots,m, \\
   x(t) &= x+ \int_0^t \theta(s) ds. 
 \end{align*}
Here $F_{(i_1,\dots, i_k)} = F_{i_1} F_{i_2} \dots F_{i_k} \theta$, i.e. $F_I \theta$ flips all components of $\theta$ listed in the tuple $I = (i_1, \dots, i_k)$.
 This defines a piecewise constant trajectory $\theta(t)$ 
 such that at at time $\tau_k$, the $i_k$\textsuperscript{th} component
 of $\theta(t)$ changes sign. The final position $(x(\tau_{m+1}), \theta(\tau_{m+1}))$
 will be denoted by $\Phi_{\mathbf{u}}(x,\theta)$. 
 
The following definitions apply for switching intensities $\lambda_i(x,\theta)$ satisyfing~\eqref{eq:relation_lambda}.

 \begin{definition}[Flippability]
   A component $i$ of the velocity is \emph{flippable} at a point $(x,\theta)\in E$
   if the corresponding switching rate $\lambda_i(x,\theta)$ is strictly positive. 
   \end{definition}

 \begin{definition}[Admissible controls]
   Given a starting point $(x,\theta)$, a control sequence $(\mathbf{t}, \mathbf{i})$ is \emph{admissible}
   if $i_k$ is flippable at the point $(x(\tau_k), \theta(\tau_k))$, that is, if 
   \[ \forall k\in\{1,\dots,m\}, \quad \lambda_{i_k} (x(\tau_k), \theta(\tau_k)) > 0. \]
   \end{definition}
   \begin{definition}[Reachability]
     Given a starting point $(x,\theta)$ and an end point $(x',\theta')$, we 
    say that  $(x',\theta')$ is \emph{reachable} from $(x,\theta)$ and we write $(x,\theta)\leadsto(x',\theta')$
    if there exists an admissible control sequence 
     $\mathbf{u}=(\mathbf{t},\mathbf{i})$ such that  $\Phi_\mathbf{u}(x,\theta) = (x',\theta')$. 

     We write $(x,\theta)\fullyleadsto(x',\theta')$ if in addition,  
     every index in $\{1,...,d\}$ appears at least once in $\mathbf{i}$, that is, 
     all the components of the velocity are flipped at least once during the trajectory. 
   \end{definition}
   
 Our goal in this section is to prove that, under weak assumptions, 
 any point is reachable from any other point. 
  It is clear that if $(x,\theta)\leadsto (y,\eta)$ using the
  canonical, minimal switching rates $\lambda_i(x,\theta) = (\partial_i U(x) \theta_i)_+$, then the same is true for any choice of the switching rates.
  Consequently, we may and will assume in this section 
 that the $\lambda_i$ are the canonical switching rates.

\begin{remark}
\label{rem:admissible_on_open_set}
It follows immediately that if $(\mathbf t, \mathbf i)$ is an admissible
control sequence for some initial configuration,  then by continuity of $\lambda$ there exists an open environment $U$ of $\mathbf t \in (0,\infty)^{m+1}$ such that $(\tilde {\mathbf t}, \mathbf i)$ is admissible for
the same initial configuration, for any $\tilde {\mathbf t} \in U$. 
\end{remark}

\begin{remark}[Reachability is transitive]
  Given two control sequences $\mathbf{u}=(s_0,\dots,s_p ; i_1,\dots,i_p)$ 
  and $\mathbf{v} = (t_0,\dots,t_q ; j_1,\dots,j_q)$, we can concatenate them
  into 
  \[ \mathbf{w} = (s_0,\dots,s_{p-1},s_p + t_0,t_1,\dots,t_q ; i_1,\dots,i_p,j_1,\dots,j_q).\]
  If $\mathbf{u}$ is admissible starting from $(x,\theta)$ and $\mathbf{v}$ 
  is admissible starting from $\Phi_\mathbf{u}(x,\theta)$, then 
  $\mathbf{w}$ is admissible starting from $(x,\theta)$ and 
  $\Phi_\mathbf{w}(x,\theta) = \Phi_\mathbf{v}\circ \Phi_\mathbf{u}(x,\theta)$. 
\end{remark}

\begin{remark}[Time reversal]
  \label{rem:time_reversal}
  If $(x,\theta)\leadsto (x',\theta')$, then $(x',-\theta')\leadsto (x,-\theta)$: 
  indeed if $\lambda_i(x,\theta) > 0$ then 
  \[ \lambda_i(x,-F_i(\theta)) = (\theta_i \partial_i U(x))_+ = \lambda_i(x,\theta) > 0,\]
  so if $(t_0,\dots,t_m; i_1,\dots,i_m)$ is an admissible control 
  that sends $(x,\theta)$ to $(x',\theta')$, then the reversed 
 sequence  $(t_m,\dots,t_0 ; i_m,\dots,i_1)$ 
  is admissible and sends $(x',\theta')$ to $(x,\theta)$. (We thank the AE for pointing out that, without further conditions, this does not hold for non-canonical switching intensties.)
\end{remark}

 We will first establish reachability for the case where the potential
 $U$ is quadratic, so that the target measure is Gaussian. We will 
 use this in Section~\ref{sub:local_min} to see that around a local 
 minimum of the potential, we can reach any velocity. We will then 
 show that, under Growth Condition~\ref{ass:GC1}, starting from any point, 
 it is possible to switch all components of the velocity. All these 
 results will be put together in Section~\ref{sub:reachability_general} to 
 prove reachability in the general case.

\subsection{Reachability for multivariate normal distributions}
\begin{proposition}
  \label{prop:gaussian_reachability}
  Suppose that the target distribution is a nondegenerate Gaussian
  $U(x) = \scal{x,Ax}$, where $A$ is a positive definite symmetric matrix. Then 
  for any $(x,\theta)$, $(x',\theta')$, $(x,\theta)\leadsto(x',\theta')$. 
\end{proposition}
Even for this simple case, the fact that the jump rates may be zero and that the process may
be unable to jump for long stretches makes the proof quite involved. The main idea is 
to use the fact that by going in a straight line for a sufficiently long time, the 
process will always reach a region where it can switch some components of its velocity. 
Let us first define a useful notational shortcut. 
\begin{definition}
  [Reachability for velocities]
  For any two velocities $\theta$, $\theta'$, we say that $\theta'$ 
  is \emph{reachable} from $\theta$ and we write  $\theta\leadsto \theta'$
  if for any $x$, there exists an $x'$ such that $(x,\theta) \leadsto (x',\theta')$. 
\end{definition}
\begin{definition}[Asymptotic flippability]
  Let $\theta\in\{-1,1\}^d$. If $\sum_j \theta_i A_{ij} \theta_j>0$ we say that the $i$\textsuperscript{th}
  component of $\theta$ is \emph{asymptotically flippable}. 
  The velocity $\theta$ itself is called asymptotically flippable 
  if all its components are asymptotically flippable. 
\end{definition}
The above definition is explained by noting that in case of asymptotic flippability of the $i$-th component, along any trajectory $x + \theta t$ the $i$-th switching intensity will eventually become positive.

\begin{lemma}
  \label{lem:asymptotically_flippable:velocities}
  If $I$ is a sequence of asymptotically flippable components for $\theta$, then $\theta \leadsto F_I(\theta)$. 
  In particular, if $\eta$ is asymptotically flippable, then for any  $\theta$, $\eta\leadsto \theta$. 
\end{lemma}
\begin{proof}
  Write $I = \{i_1, \dots, i_m\}$. Starting from $x$ with 
  velocity $\theta$, after a large time $t$ the components of $A(x+t\theta)$ 
  will have the signs of the components of $A\theta$, so the 
  $i$\textsuperscript{th} component for $i\in I$ will all be flippable. 
  The ``pseudo''-control sequence $(t,0,\dots, 0 ; i_1,\dots, i_m)$, would therefore
  bring $(x,\theta)$ to $(x',F_I\theta)$ for some $x'$. It is strictly speaking not a control sequence since its times between switches are zero. However since the positivity of the 
  jump rates is an open condition and the map $\mathbf{t}\mapsto \Phi_{(\mathbf{t},\mathbf{i})}(x,\theta)$
  is continuous, this implies the existence of a $\mathbf{t'}$ with positive 
  coefficients such that $(\mathbf{t'} ; i_1,\dots,i_m)$ is admissible starting from $(x,\theta)$, 
  proving that $\theta\leadsto F_I(\theta)$. 
\end{proof}
  
The usefulness of this definition is readily seen through the following result. 
  
\begin{lemma}
  [Reachability for asymptotically flippable velocities]
  \label{lem:asymptotically_flippable:positions}
  If $\eta$ is asymptotically flippable, then for any $x$ and $x'$, $(x,\eta)\leadsto (x',-\eta)$.  
\end{lemma}
Before proving this lemma, let us give a simple case where
it   is enough to conclude the argument. 
\begin{corollary}
  If $A$ is diagonally dominant, then every $\theta$ is asymptotically flippable,
  and $(x,\theta)\leadsto(x',\theta')$ for all pairs of states. 
\end{corollary}
\begin{proof}If $A$ is diagonally dominant then 
  $\sum_j \theta_i A_{ij} \theta_j  \geq A_{ii} - \sum_{j, j\neq i} \abs{A_{ij}} >0$
  so all velocities are asymptotically flippable. 
  Given $(x,\theta)$ and $(x',\theta')$, we first use Lemma~\ref{lem:asymptotically_flippable:velocities}
   to get the existence of $x''$ such that $(x,\theta)\leadsto (x'',-\theta')$.
  By Lemma~\ref{lem:asymptotically_flippable:positions} we can then reach $(x',\theta')$ from $(x'',-\theta')$, and we are done by transitivity. 
\end{proof}
\begin{proof}[Proof of Lemma~\ref{lem:asymptotically_flippable:positions}]
  Let $\eta$ be an asymptotically flippable velocity, and $x$, $x'$ be two arbitrary
  positions. 
  To control the system from $x$ to $x'$, the idea is to go very far in the direction 
  of $\eta$, to a region where all components of $\eta$ are flippable, to flip them 
  in a well chosen order and with well chosen time intervals between flips, so that
  when the last component is flipped, the system reaches $x'$ after a long run 
  in the direction $-\eta$.

  To do this rigorously, define $d_i = (x'_i - x_i)/\eta_i$, and suppose 
  first that the $d_i$ are increasing: 
  $d_1 < \dots < d_n$. For $1\leq i \leq n-1$, 
  let $t_i = (d_{i+1} - d_{i})/2$, and choose $t_0$ and $t_n$ positive 
  numbers such that $t_0-t_n = \frac{d_1+d_n}{2}$. 

  Now let $t$ be a large time to be chosen later, and consider the control 
  \[
    (\mathbf{t},\mathbf{i}) = (t+t_0,t_1,\dots,t_{n-1},t_n+t ; 1,2,\dots, n).
  \] 
  Starting from $(x,\eta)$, the $i$\textsuperscript{th} component of 
  the position will follow $\eta_i$ for a time $t+t_0+\cdots+ t_{i-1}$, 
  and $-\eta_i$ for the remaining time $t_i + \cdots + t_n +t$. Therefore, 
  the $i$\textsuperscript{th} component of the final position 
  is 
  \begin{align*}
    &x_i + \eta_i(t+\sum_{j=0}^{i-1} t_j) - \eta_i (t + \sum_{j=i}^n t_j) \\
    &\quad = x_i + \eta_i (t_0 - t_n + \frac{1}{2}\sum_{j=1}^{i-1} (d_{j+1} - d_j)
                         - \frac{1}{2}\sum_{j=i}^{n-1} (d_{j+1} - d_j)) \\
    &\quad = x_i + \frac{\eta_i}{2}\PAR{ d_1 + d_n + d_i -d_1 - d_n + d_i} \\
    &\quad = x_i  + x'_i - x_i = x'_i. 
  \end{align*}
  If the $d_i$ are not increasing but all distinct, we can reorder them by
  finding a permutation $\sigma$ such that the $d_{\sigma(i)}$ increase, and
  perform the same argument using the control sequence
  $(t+t_0,t_1,\dots,t_{n-1},t_n+T ; \sigma(1), \dots, \sigma(n))$ where $t_i =
  (d_{\sigma(i+1)} - d_{\sigma(i)})$. 

  It remains to check that all the moves are admissible. By a computation
  similar to the one just above, the position $x^{(i)}$ just before the $i$\textsuperscript{th} 
  flip in the control sequence is given by: 
  \[ x^{(i)}_j = x_j + \eta_j(t+ \sum_{k=0}^{i-1} (\ind{k\leq j} - \ind{k>j}) t_k).\]
  Once the $t_k$ are fixed (by the given input of the starting and ending positions $x$ and $x'$), 
  one can always take $t$ large enough so that $(Ax^{(i)})_i$ has the sign of $\eta_i$, which 
  implies that the $i$\textsuperscript{th} jump is indeed admissible. 

  Finally, if some of the $d_i$ are equal, we may always introduce 
  intermediary points $y$ and $y'$ such that the differences $(y_i-x_i)/\eta_i$ 
  are distinct for all $i$, and likewise the differences $(y'_i-y_i)/(-\eta_i)$, 
  and $(x'_i - y_i)/\eta_i$. Therefore $(x,\eta)\leadsto (y,-\eta)\leadsto(y',\eta)
  \leadsto(x',-\eta)$, and we are done by transitivity. 
\end{proof}

We now tackle the general case, when  $A$ is not diagonally dominant. 
\begin{lemma}[All roads lead to an asymptotically flippable velocity]
  \label{lem:all_roads}
  For all $\theta$ there exists an asymptotically flippable velocity $\eta$ such that
  $\theta\leadsto \eta$. 
\end{lemma}
\begin{proof}
  To prove this result, it is useful to represent the matrix $A$ as a Gramian matrix: 
  as can be seen by an $LL^\top$ or a symmetric square root representation, there exists
  a family of vectors $(v_1,\dots,v_n)$ such that $A_{ij} = \scal{v_i,v_j}$. For a velocity
  $\theta$, let $v(\theta) = \sum_i \theta_i v_i$. Using this representation, 
  we have the equivalence:
  \begin{align*}
    i \text{ is asymptotically flippable for $\theta$} 
    &\iff (A\theta)_i\theta_i >0 \\
    &\iff \scal{\theta_i v_i, v(\theta)} >0. 
  \end{align*}
  Let $\theta$ be an arbitrary velocity, and suppose that $\theta$ is not asymptotically flippable. 
  Denote by $I$ the subset of asymptotically flippable indices: 
  \[ i\in I \iff \scal{\theta_i v_i, v(\theta)}>0. \]
  Since $\sum_i \scal{\theta_i v_i,v(\theta)} = \abs{v(\theta)}^2 >0$ by positive definiteness of $A$, this set
  is non empty; by hypothesis it is not equal to $\{1,\dots,n\}$. 
  Let $F_I(\theta)$ be the velocity obtained by flipping all asymptotically flippable components. 
  The key point is that this flip increases the norm of $v$: 
  \[ \abs{v(F_I(\theta))} > \abs{v(\theta)}.\]
  Indeed, let $v_+ = \sum_{i\in I} \theta_i v_i$ and $v_- = \sum_{i\notin I} \theta_i v_i$. 
  Since $v(\theta) = v_+ + v_-$ and $v(F_I \theta) = v_- - v_+$, 
  \[
    \abs{v(F_I \theta)}^2 - \abs{v(\theta)}^2 = -4 \scal{v_-,v_+}.
  \]
  Now $\scal{v(\theta),v_-}$ must be non-positive by definition of $v_-$ and the set $I$, 
  but this is $\abs{v_-}^2 + \scal{v_-,v_+}$. The scalar product $\scal{v_-,v_+}$ 
  is therefore negative, and 
  \[ 
    \abs{v(F_I \theta)} >  \abs{v(\theta)}. 
  \]
  
  Now starting from $\theta$, apply the following `algorithm':
\begin{itemize}
 \item if $\theta$ is asymptotically flippable, stop.
 \item if it is not, move to $F_I \theta$ where $I$ is the set of asymptotically flippable indices.
\end{itemize}

The fact that $\theta$ is not asymptotically flippable implies that $v_-$ cannot be zero
(because $I\neq \{1,\dots,d\}$ and the $v_i$ are linearly independent because
$A$ is positive definite), so the norm will increase. Since along the
algorithm, $|v(\theta)|$ is strictly increasing, it must stop at one time; at
this time it has (by definition) reached an asymptotically flippable velocity.
\end{proof}

Now we have all the ingredients to prove the full reachability in the Gaussian case. 

\begin{proof}[Proof of Proposition~\ref{prop:gaussian_reachability}]
  Let $(x,\theta)$ and $(x',\theta')$ be two points. 
  By Lemma~\ref{lem:all_roads}, there exists an asymptotically flippable velocity $\eta'$ and a point $y'$
  such that $(x',-\theta')\leadsto (y',\eta')$. By the time-reversal property of Remark~\ref{rem:time_reversal},
  $(y',-\eta')\leadsto (x',\theta')$. 
  Now by Lemma~\ref{lem:all_roads} again, we get the existence of an
  asymptotically flippable velocity $\eta$ and a point $y$ such that
  $(x,\theta)\leadsto(y,\eta)$. 
  Lemma~\ref{lem:asymptotically_flippable:velocities} gives us a point $z$ such that
  $(y,\eta) \leadsto (z,\eta')$, and Lemma~\ref{lem:asymptotically_flippable:positions}
  tells us that $(z,\eta') \leadsto (y',-\eta')$, which finishes the
  construction of an admissible trajectory.
\end{proof}

\subsection{Reachability around a local minimum}
\label{sub:local_min}

As before $U : \R^d \rightarrow \R$ is the potential function of a probability
density function $\overline \pi$, i.e.  $\overline \pi(x) \propto \exp(-U(x))$. We suppose
that $U$ has at least one nondegenerate local minimum, which we assume without
loss of generality to be located in $x = 0$, i.e. $\nabla U(0) = 0$ and $V :=
H_U(0)$ is positive definite. We will use the fact that all points in $\R^d$
are reachable through zigzag trajectories for the Gaussian density
$\pi^V \propto \exp(-\tfrac 1 2 x^T V x)$, to conclude that the same holds in a
neighbourhood of 0 for the potential $U$.

\begin{lemma}
\label{lem:nondegeneratelocalminimum}
Suppose $U \in \mathcal C^3(\R^d)$, $\nabla U(0) = 0$ and $H_U(0)$ is positive
definite. There exists a radius $\gamma > 0$  such that 
$(x,\theta) \fullyleadsto (y,\eta)$ for every $(x,\theta)$ and $(y,\eta)$ satisfying $|x| < \gamma$, $|y| < \gamma$.
\end{lemma}

\begin{proof}
Let the switching rates for the Gaussian density $\pi^V$ be denoted by
$(\lambda_i^V)$. For a given control sequence $(\mathbf t, \mathbf i) = (t_0, \dots,
t_p; i_1, \dots, i_p)$  with associated switching points $(x(\tau_i),
\theta(\tau_i))_{i=1}^p$ and final point $(x(\tau_{p+1}),\theta(\tau_{p+1}))$,
define
\[ \lambda_{\min}^V(\mathbf t, \mathbf i) 
  = \min_{j=1,\dots,p} \lambda^V_{i_j}(x(\tau_j), \theta(\tau_j)) 
  \quad \text{and} \quad
  r_{\max}(\mathbf t, \mathbf i)
  = \max_{j=0, \dots, p+1} |x(\tau_j)|,
\]
for the minimum switching rate at a switching point and maximum distance from
the origin for the associated trajectory, respectively.  For $n \in \N$ and
$\theta, \eta \in \{-1,1\}^d$ define sets 
\begin{multline}
    \mathcal U_{n,\theta,\eta} 
   := \{ y \in \R^d: |y| < 2, (0,\theta) \fullyleadsto (y,\eta), \ \text{through a control $(\mathbf t, \mathbf i)$ such that} \\
\text{$\lambda_{\min}^V(\mathbf t, \mathbf i) > 1/n$ and $r_{\max}(\mathbf t, \mathbf i) < n$}\}. 
\end{multline}

Suppose $y \in \mathcal U_{n, \theta, \eta}$, so that there exists a control
$(\mathbf t, \mathbf i)$ taking $(0,\theta)$ to $(y,\eta)$ by which every
component of the direction vector is flipped. By perturbing the switching times
$t_1,\dots, t_p$ in the control, we find that $(0, \theta) \leadsto (y',\eta)$
for all $y'$ in a sufficiently small neighbourhood of $y$ through a control
$(\mathbf t', \mathbf i')$ such that $\lambda_{\min}^V(\mathbf t', \mathbf i')
> 1/n$ and $r_{\max}(\mathbf t', \mathbf i') < n$. It follows that $\mathcal
U_{n, \theta, \eta}$ is open for all $n, \theta, \eta$.
For a Gaussian density we have 
$(x, \theta) \fullyleadsto (y,\eta)$ for all $(x,\theta), (y,\eta) \in E$ by a repeated use of Proposition~\ref{prop:gaussian_reachability}. Thus for fixed $\theta, \eta$ we have the following
open cover of the closed unit disc $D = \{ y \in \R^d : |y| \leq 1\}$:
\[ 
  D \subset \bigcup_{n \in \N} \mathcal U_{n, \theta, \eta}.\]
By compactness of $D$, for all $\theta,\eta$, there exists an $N_{\theta,\eta} \in \N$ such that
\[
  \{ y \in \R^d : |y| \leq 1\} \subset \mathcal U_{N_{\theta,\eta}, \theta, \eta}.
  \]
Let $N := \max_{\theta, \eta} N_{\theta, \eta}$. It follows that for every
$\theta \in \{-1,1\}^d$ and $(y,\eta) \in E$, $|y| \leq 1$, we have
$(0,\theta) \fullyleadsto (y, \eta)$
through
trajectories with minimal switching rate larger than $1/N$ and a maximal
distance from the origin smaller than $N$.
By a Taylor expansion we have that, for some constant $c$, which we may assume
to satisfy $c >  1$, 
\begin{equation} 
  \label{eq:difference-from-gaussian}
  |\nabla U(x) - V x| \leq  c |x |^2 \quad \text{for $|x| \leq 1$}.
\end{equation}
Now let $\theta \in \{-1,1\}^d$ and $(y,\eta) \in E$,
such that $|y| < \gamma := \frac 1 {2 c N^3}$. Let $z = y/\gamma$ so that  $|z|
< 1$. There exists a control sequence $(\mathbf t, \mathbf i)$ for which $(0,\theta)
\fullyleadsto (z,\eta)$ such that $\lambda_{\min}^V(\mathbf t, \mathbf i) > \frac
1 N$ and $r_{\max}(\mathbf t, \mathbf i) < N$.  After a rescaling of $\mathbf
t$ to $\mathbf t' = \gamma \mathbf t$ we obtain a control sequence for
$(0,\theta) \fullyleadsto (y,\eta)$ such that $\lambda_{\min}^V(\mathbf t',
\mathbf i) > \frac \gamma N = \frac 1 {2c N^4}$ (since the switching rates for
the Gaussian potential scale linearly with distance from the origin), and such
that the complete trajectory is contained within a ball of radius $\gamma N <
\frac 1 {2 c N^2} < 1$, so that we may
apply~\eqref{eq:difference-from-gaussian} along the trajectory. Along the trajectory with switching times $(\tau_j)_{j=1}^p$
corresponding to the control sequence $(\mathbf t',\mathbf i)$, we obtain
\[ 
  |\nabla U(x(\tau_j)) - V x(\tau_j)| 
  \leq c |x(\tau_i)|^2 < c \gamma^2 N^2
  = \tfrac 1 {4cN^4},
  \quad j = 1, \dots, p,
\]
so that, for all $j = 1, \dots, p$,
\[
  \lambda_{i_j}(x(\tau_j),\theta(\tau_j))
  = (\theta(\tau_j) \partial_{i_j} U(x(\tau_j)))_+ 
  \geq \lambda^V_{i_j}(x(\tau_j), \theta(\tau_j)) - \tfrac 1 {4cN^4} 
  > \tfrac 1 {4 cN^4} > 0,
\]
i.e. the control sequence $(\mathbf t', \mathbf i)$ is admissible for
$(0,\theta) \fullyleadsto  (y,\eta)$ with respect to the switching rates
$(\lambda_i)$. 

By an analogous argument there exists an admissible control sequence for
$(y,\eta) \fullyleadsto (0,\theta)$. The statement of the proposition follows by
concatenation of trajectories.
\end{proof}

\subsection{Flippability}
Recall that $(x,\theta) \fullyleadsto (y,\eta)$ if there 
is an admissible path from $(x,\theta)$ to $(y,\eta)$  along which all components of the velocity are switched. 
\begin{definition}[Full flippability]
  \label{def:full_flip}
  The process is \emph{fully flippable} if  for each $(x,\theta)$, there exists a $(y,\eta)$
  such that $(x,\theta)\fullyleadsto (y,\eta)$, 
\end{definition}

\begin{proposition}
  \label{prop:full_flip}
  If the potential $U$ satisfies Growth Condition~\ref{ass:GC1}, then the 
  process is fully flippable. 
\end{proposition}

\begin{proof}
  By definition, the process is fully flippable if for all points 
  $(x,\theta)$, there exists an admissible control sequence $(\textbf{i},\textbf{t})$ 
  such that all indices appear in $\textbf{i}$. Striving for a contradiction, 
  suppose that there is an $(x,\theta)$ such that, for  any admissible control sequence, 
  there is an index in $\{1,...,d\}$ that does not appear in the indices sequence. 
  Suppose that  starting from $(x,\theta)$, 
  we are able to construct, for any $\eps$ and any $T$, 
  an admissible trajectory $(x(t),\theta(t))_{t\in [0,T]}$ along
  which the following bound holds: 
  \begin{equation}
    \label{eq:niceCS}
    \forall i, \forall t\in[0,T], \quad 
    \theta(t) \partial_i U (x(t)) < \eps. 
  \end{equation}
  Integrating $U$ along this trajectory, we get $U(x(T)) \leq U(x) + \eps d
  T$. However, by hypothesis this trajectory leaves at least one index in
  the velocity unchanged, so $\| x(T) -  x\|_\infty \geq T$. This shows that
  \[ 
    \inf\{ U(y): y \text{ such that } \|y-x\|_\infty \geq T\}\leq U(x) + \eps dT,
\]
and is therefore not larger than $U(x)$ by taking $\eps$ to zero. This contradicts
  the hypothesis that $U$ converges to infinity. 
 
  Let us now prove that such trajectories exist. Fix $\eps>0$, 
 and  say $T$ is ``nice'' if there exists an admissible control sequence
 starting from $(x,\theta)$ such that the bound \eqref{eq:niceCS} holds. 
  The set of nice $T$ is clearly open in $[0,\infty)$, so it will be enough 
 to check that it is closed. 

 To this end, suppose that the $T_n$ are an increasing sequence of nice times
  converging to $T$. The natural idea to construct a nice trajectory of length 
  $T$ is to pick a trajectory of length $T_n$ and continue it in the 
  final direction $\theta_{T_n}$ until time $T$. The corresponding trajectory
  will be admissible, but it may fail to satisfy \eqref{eq:niceCS} if, during
  the interval $[T_n,T)$, one of the quantities $(\theta(t))_i \partial_i U(x(t))$ crosses
  the level $\eps$. We will prove that by switching the corresponding indices, 
  we can construct a nice trajectory. 
  
  Since the process moves at finite speed, we know that 
  all admissible trajectories of length less than $T$ starting from $(x,\theta)$ 
  will lie in a bounded set, only depending on $T$. Let $C_T$ be an upper
 bound on the Hessian of $U$ on this bounded set. Let $n$ be large enough 
so that $T - T_n < \eps/{2C_T}$, and consider a ``nice'' trajectory of length $T_n$;
we wish to continue it up to time $T$. 
Let $D=\{i_1,...,i_m\}$ be the set of ``dangerous'' indices, that is, indices for which 
$\theta_i \partial_i U(x(T_n)) > \eps/2$. Consider the trajectory obtained by concatenating
the nice $T_n$ control sequence with the sequence $(i_1,...,i_m ; \eps',...,\eps', T-T_n - m\eps')$. 
If $\eps'$ is small enough, this trajectory will be both admissible and nice: 
all ``dangerous'' indices will be switched before the corresponding product reaches $\eps$, 
and they will not have time to grow up to $\eps$ again. The set of nice $T$ is therefore 
$[0,\infty)$ in its entirety. 
\end{proof}

\subsection{Reachability in the general case}
\label{sub:reachability_general}

\begin{lemma}
  \label{lem:from_switches_to_openness}
  If $(x,\theta)\fullyleadsto (y,\eta)$,  then there is an open neighborhood 
  $U$ of $(y,\eta)$ such that for all $(y',\eta')\in U$, $(x,\theta)\fullyleadsto (y',\eta')$. 
\end{lemma}
\begin{proof}
  By hypothesis there is a sequence of times and indices such that 
  \[ y = x + t_0 \theta + t_1 F_{i_1}\theta + \cdots t_n F_{i_1,...,i_n}\theta. \]
  Define $\Phi:(s_0,...,s_n)\mapsto  x + s_0 \theta + s_1 F_{i_1}\theta + \cdots s_n F_{i_1,...,i_n}\theta$. 
  Then $D\Phi = (\theta,F_{i_1}\theta,...,F_{i_1,...,i_n} \theta)$. 
  Since the difference between two consecutive vectors in this family
  is $\pm 2 e_{i_k}$, the map $\Phi$ has full rank if all components are switched at least once. 
  Therefore $\Phi$ is a submersion from a neighborhood of $(t_0,...,t_n)$ to 
  a neighborhood of $y$. By continuity of the switching rates, we may assume without loss of
  generality that for all $(s_0,...,s_n)$ in this neighborhood, the corresponding
  trajectory is admissible. Since the sequence of switches is the same as the original trajectory, 
  we get the conclusion. 
\end{proof}

Say $(x,\theta) \sim (x',\theta')$ if they are equal or if $(x,\theta)\leadsto(x',\theta')\leadsto(x,\theta)$. 
Denote by $\cl{x,\theta}$ the equivalence class of $(x,\theta)$
and by $R$ the velocity reversal (applied to points in, or subsets of, $\dR^d\times \{-1,1\}^d$). 
\begin{lemma}
  \label{lem:alternativeForClasses}
  The equivalence classes of $\sim$ are either a single point or an open set in $\dR^d\times \{-1,1\}^d$. 

  For any $(x,\theta)$, $R(\cl{x,\theta}) = \cl{ R(x,\theta)}$. In particular, the
  classes of $(x,\theta)$ and $(x,-\theta)$ have the same type (open or singleton). 
\end{lemma}
\begin{proof}
  Suppose that  $(x,\theta)$ and $(x',\theta')$ are two different equivalent points.
  This means that there is an admissible loop starting from, and returning to,
  $(x,\theta)$.  Along such a loop all components of the velocity must be
  flipped at least once: if the $i$\textsuperscript{th} component of the
  velocity stays at $1$ (resp. $-1$),  then the $i$\textsuperscript{th}
  component of the position strictly increases (resp. decreases) along the
  loop, a contradiction. Therefore if $\cl{x,\theta}$ is not a singleton, then 
  $x\fullyleadsto x$. 
  
  Let us now prove openness. If $(y,\eta)$ is in the non-trivial 
  class of $(x,\theta)$, then $(x,\theta)\fullyleadsto (x,\theta)\leadsto(y,\eta)$, 
  so $(x,\theta)$ leads to all points near $(y,\eta)$. Similarly, 
  $(x,-\theta)\fullyleadsto (y,-\eta)$, so $(x,-\theta)$ leads to all 
  points in a neighborhood of $(y,-\eta)$, and by reversal, 
  all points near $(y,\eta)$ must lead to $(x,\theta)$. Therefore
  all points near $(y,\eta)$ are in fact equivalent to $(x,\theta)$ 
  and the class is open. 

  The reversal property is a consequence of the similar property for $\leadsto$. 
\end{proof}

\begin{proposition}
  [Stability of open classes]
  \label{prop:stabilityOfOpenClasses}
  The open equivalent classes are ``almost stable'' under $\leadsto$ and its inverse, that is, 
  if the class of $(x,\theta)$ is open, then for $\pi$-almost every $(y,\eta)$, 
  we have the equivalence  $(y,\eta) \leadsto (x,\theta) \iff (x,\theta)\leadsto(y,\eta)
  \iff (x,\theta)\sim (y,\eta)$.

  If the process is fully flippable in the sense of Definition~\ref{def:full_flip},
  then the open classes 
  are of the form $\dR^d\times V$, where $V$ is a subset of the velocities $\{-1,1\}^d$. 
\end{proposition}
\begin{remark}[Terminology]
  In the countable state setting, classes that are stable under the 
 analogue of $\leadsto$ are  called ``essential'' (see, e.g., \cite{LPW09}).
 In a general state space, it is known that the communication structures 
 are more difficult to define and study; this has led in particular 
to the definition of $\psi$-irreducibility,  see \cite[Chapter 5]{MT09}. 
It turns out that in our particular case, the relation $\leadsto$ defines 
interesting equivalence classes that we can study before discussing 
$\psi$-irreducibility. 
\end{remark}
\begin{proof}
  The first step is probabilistic. 

  Let $O$ be an open class. Let $O_+$ be the ``future'' of $O$, that is, 
  the set of $(y,\eta)$ such that there exists $(x,\theta)\in O$ such 
  that $(x,\theta)\leadsto (y,\eta)$. Note that since $(x,\theta)\fullyleadsto (x,\theta)$, 
  $O_+$ is open, therefore measurable. Let $P^t((x,\theta),A)$ denote the Markov transition kernel of the zigzag process. Let us use the invariance of $\pi$ 
  through the resolvent kernel: 
  \begin{align*}
    \pi(O_+) &= \int_0^\infty e^{-t} \pi P^t(\cdot, O_+) dt \\
    &= \int_0^\infty \int_{E} e^{- t} \prb[(x,\theta)]{ (X_t,\Theta_t)\in O_+} d\pi(x,\theta) dt \\
    &= \int_0^\infty \int_{E} e^{- t} \ind{(x,\theta)\in O_+} \prb[(x,\theta)]{ (X_t,\Theta_t)\in O_+} d\pi(x,\theta) dt \\
    &\quad   + \int_0^\infty \int_{E} e^{- t}  \ind{(x,\theta) \notin O_+} \prb[(x,\theta)]{ (X_t,\Theta_t)\in O_+} d\pi(x,\theta) dt.
  \end{align*}
  Since $O_+$ is stable by $\leadsto$, the probability in the first integral is $1$, so the whole first integral 
  is equal to $\pi(O_+)$. Therefore the second integral must vanish: for all $(x,\theta)$ in some 
  set $A$ of full $\pi$-measure, 
  \[ \ind{(x,\theta)\notin O_+} \int_0^\infty e^{-t} \prb[(x,\theta)]{ (X_t,\Theta_t)\in O_+} dt = 0.\]
  If $(x,\theta)$ is in $A$  and leads to a point in $O$, then the 
  probability above is strictly positive, so $(x,\theta)$ must be in $O_+$. Consequently we can 
  build a loop from $(x,\theta)$ that intersects $O$, so $(x,\theta)$ is in $O$. 

  In the other direction, we use reversal. Without loss of generality we may assume 
  $A$ is stable by reversal of velocities. If $(x,\theta)$ in $A$ is reachable from 
  a point $(y,\eta)$ in $O$, then $(x,-\theta)\leadsto(y,-\eta)$, so $(x,-\theta)\in RO$, 
  and $(x,\theta)\in O$.

  We now prove a stronger stability statement by  getting rid of the ``$\pi$-almost surely''. 
  Consider a point $(x,\theta)$ in an open class $O$ and 
  suppose that $(y,\eta)$ is reachable from $(x,\theta)$. By 
  the assumption, we can find a $(z,\xi)$ such that $(y,\eta)\fullyleadsto (z,\xi)$. 
  By Lemma~\ref{lem:from_switches_to_openness}, $(y,\eta) \leadsto (z',\xi')$ for all $(z',\xi')$ 
  in a neighborhood of $(z,\xi)$. By transitivity, $(x,\theta)$ itself leads
  to all points in this neighborhood. Such a neighborhood must have a positive 
  $\pi$-measure, so at least one of the $(z',\xi')$ leads back to $(x,\theta)$. 
  Therefore we have a loop $(x,\theta)\leadsto (y,\eta)\leadsto (z',\xi')\leadsto(x,\theta)$, 
  so all three points are in the same class, so open classes are stable by $\leadsto$. 
  Using reversal it is easy to see that they are also stable in the other direction.

  The third step of the proof is to use the stability to prove that non-trivial classes
  are closed, and must therefore consist of a certain number of copies of $\dR^d$. Let 
  $O$ be a non-trivial class, and let $(x,\theta)$ be a point in the (topological) 
  closure of $O$. By Lemma~\ref{lem:from_switches_to_openness}, there exists 
  a $(y,\eta)$ and an open set $\mathcal{U}$ such that  $(x,\theta)$ 
  leads to all points in $\mathcal{U}$. 
  Write
  $y = x+ t_0 \theta + \cdots + F_{i_1,...,i_n} \theta$ for some sequence 
  of times and indices. By continuity of the switching rates, 
  the same control sequence will be admissible if $x'$ is close to $x$, and will 
  lead from $(x',\theta)$ to the point $(y',\eta) = (y+x'-x,\eta)$. 
  Since $x$ is in the closure of $O$, we can find $x'$ in $O$ such 
  that $(x',\theta)\leadsto (y',\eta)$, and we may assume that $(y',\eta)$ is
  in $\mathcal{U}$, so that $(x,\theta)\leadsto(y',\eta)$. 
  Since $(x',\theta)$ is in $O$, $(y',\eta)$ is also in $O$ by forward stability, 
  so $(x,\theta)$ is itself in $O$, proving that $O$ is closed. 
\end{proof}

\begin{theorem}
  \label{thm:reachability}
  If the potential $U$ is $\mathcal C^3$, satisfies Growth Condition~\ref{ass:GC1}, and has a nondegenerate local minimum, 
  then there is only one equivalence class. In particular $(x,\theta) \leadsto (y,\eta)$ for all $(x,\theta) \in E$ and $(y,\eta) \in E$.
\end{theorem}
\begin{proof}
  By the local minimum approximation result (Lemma~\ref{lem:nondegeneratelocalminimum}), we know that there exists an 
  open set $\mathcal{U}$ such that all points in $\mathcal{U}\times\{-1,1\}^d$
  are in the same equivalence
  class, say $O$. By Lemma~\ref{lem:alternativeForClasses}, $O$ must then be open. 
  Since the potential $U$ goes to infinity, the process is fully flippable by Proposition~\ref{prop:full_flip}, 
  so we may apply Proposition~\ref{prop:stabilityOfOpenClasses} to see that $O$ consists of
  copies of $\dR^d$. Since $O$ contains $\mathcal{U}\times\{-1,1\}^d$, it 
  follows that $O = E$. 
\end{proof}

\section{Ergodicity and exponential ergodicity}
\label{sec:probabilistic}

To prove ergodicity and exponential ergodicity, we will use standard results 
from \cite{MeynTweedie1992,MT2,Twe94,MT09,DMT95}. 
In order to show that they apply, we need to check a certain number
of properties of the process. Some of these properties (aperiodicity, irreducibility)
are analogues in the continuous time and continuous space setting of classical notions for Markov chains. 
In order to guarantee that the process does not behave too wildly with 
respect to the topology of the ambient space, Meyn and Tweedie have also 
introduced the notion of $T$-processes (where $T$ stands for ``topology''). 
We will first recall these here, phrased in terms of a general Markov process $(Z_t)$ taking values in a space $E$, for completeness.
For a more detailed overview of these notions, we refer to the aforementioned papers, in particular \cite{MT2}, 
and the reference book \cite{MT09}. 

For a given measure $\psi$, a process is \emph{$\psi$-irreducible} if for any starting point
$z$ and any set $A$ of positive $\psi$-measure, $\esp[z]{\int_0^\infty \ind{A}(Z_t) \, dt}>0$. 
It is a \emph{$T$-process} if there exists a probability distribution $a$ on $\dR_+$ 
and a kernel $K(z,A)$
such that for fixed $A$, $z\mapsto K(z,A)$ is lower semi-continuous,
and for fixed $z$, $K(z,E)>0$ and we have the lower bound:
\[\int \prb[z]{Z_t \in A} \, da(t) \geq K(z,A).\]
A measurable set $C \subset E$  is called \emph{petite} if there exists a probability distribution 
$a$, a constant $c>0$ and a nontrivial measure $\nu$ on $E$ such that 
\[ \int \prb[z]{Z_t \in \cdot} \, da(t) \geq c\nu(\cdot) \quad \text{for all $z \in C$}.\]
An irreducible process is called \emph{aperiodic} if there 
exists a petite set $C$ and a time $T$ such that $\prb[z]{Z_t\in C} >0$
for all starting points $z\in C$ and all times $t\geq T$. 
The process is called \emph{Harris recurrent} if, for some $\sigma$-finite measure $\varphi$, $\P_z\left[ \int_0^{\infty} \1_{A}(Z_t) \, d t= \infty \right] \equiv 1$ whenever $\varphi(A) > 0$. As discussed in \cite{MT2}, Harris recurrence implies existence of a unique (up to constant multiples) invariant measure. If, moreover, there is a finite invariant measure (which in this paper is always the case by assumption~\eqref{eq:integrability-potential}), the process is called \emph{positive Harris (recurrent)}.

In the next sections, we  establish that the zigzag process is
in fact an irreducible, aperiodic $T$-process; Section~\ref{sec:Lyapunov}
is devoted to finding a suitable Lyapunov function.

\subsection{Continuous components}
In this section we give two results on the existence of an absolutely 
continuous component in the distribution of the position of the process. 
We start with an easy result, expressed in terms of a certain stopping time. 
\begin{lemma}[Absolute continuity from jumps]
  \label{lem:abs_cont_jumps}
  Let $(T_i)$ be the random times where the components of the velocity switch. 
  Let $N$ be the random integer such that $T_N$ is the first time when 
  $d-1$ components have switched; let $N=\infty$ if this does not occur.
  Let $\tau = T_{N+1}$ if $T_N$ is finite, and $\tau=\infty$ otherwise.

  Then the distribution of $X_\tau$ (conditionally on $\tau < \infty$) is absolutely continuous with respect 
  to the Lebesgue measure~: if $B$ is a Borel set in $\dR^d$ of Lebesgue
  measure zero, then 
  \[ \prb{ \tau<\infty, X_\tau \in B} = 0.\]
\end{lemma}

In particular, in case $d = 1$, then $N = 0$, $T_N = 0$ and $\tau$ is the time of the first switch.
\begin{proof}

 Let $B$ be a set of zero Lebesgue measure in $\dR^d$ and 
$t$ be arbitrary. It is enough to show that $\prb[(x,\theta)]{\tau \leq t ;  X_\tau \in B} = 0$, 
since this implies $\prb[(x,\theta)]{X_\tau \in B,\tau < \infty} = 0$ by monotone convergence. 

It is well known (see \cite{Benaim2015, BierkensFearnheadRoberts2016}) that the 
law of $(X_t,\Theta_t)$ may be obtained by a thinning procedure. More precisely, 
let $\overline{\lambda}$ be an upper bound on the switching rates up to time~$t$
(such a bound exists since the process has finite speed and the switching rates 
are continuous). Then the process may be constructed on $[0,t]$ by running a Poisson 
clock with intensity $\overline{\lambda}d$, and, for each Poisson event, 
picking an index $i$ uniformly, then accepting or rejecting the flip of the corresponding
component of the velocity with a probability given by $\lambda_i(x,\theta)/\overline{\lambda}$. 

Recall that $F_{i_1,...,i_k}\theta$ is the velocity obtained from $\theta$ by flipping, 
possibly many times, the components appearing in the sequence. For convenience, we 
extend this definition to allow zero values in the index sequence, which corresponds to no 
flipping.  This allows us to write
 \[
   X_\tau = x + E_1\theta + E_2 F_{I_1} \theta + ... + E_{M+1} F_{I_1,...,I_m} \theta, 
   \]
   where $M$ is a random integer (larger than $N$),  the $(I_k)$ take values 
   in $\{0,1,...,d\}$ with $I_k=j$ for $j\neq 0$ indicating a proposed and accepted $j$ flip, while $I_k=0$ corresponding to all rejected flips, and the $(E_i)$ are the interarrival times
  of the Poisson clock.  We decompose over all possible index sequences:
   \[\begin{split}
     \prb{\tau \leq t, X_\tau \in B} 
     = \sum_{m\in \dN_0} \sum_{(i_1,...,i_m)\in\{0,...d\}^m} 
     \prb{\tau \leq t, M=m, (I_1,...,I_M) = (i_1,...,i_m), 
       \right.\\
       \left. (x + E_1 \theta + \cdots + E_{m+1} F_{i_1,...,i_m}\theta) \in B}.
     \end{split}
     \]
    If $M=m$, $N\leq m$ so by definition, at least $d-1$ different (non-zero) indices must appear in 
    the sequence $(i_1,...,i_m)$, and
   \begin{align*}
     &\prb{\tau \leq t, X_\tau \in B} \\
     &\leq \sum_{m\in \dN} 
     \sum_{
       \substack{(i_1,...,i_m)\in\{0,...d\}^m\\
	 \text{$d-1$ indices appear in $(i_1,...,i_m)$}
      }}
     \prb{(x + E_1 \theta + \cdots + E_{m+1} F_{i_1,...,i_m}\theta) \in B}.
    \end{align*}
    For each term in the sum, the vectors $(\theta,F_{i_1}\theta,...,F_{i_1,...,i_m} \theta)$ 
    span $\dR^d$, so the distribution of $x+E_1\theta + \cdots + E_{m+1} F_{i_1,...,i_m}\theta$ 
    is absolutely continuous, and the probability that it falls in the 
    set $B$ is zero. 
\end{proof}

The proof of the existence of an absolutely continuous component at 
a fixed time is a bit more involved, but is the key ingredient to prove that 
the process behaves nicely. 

\begin{lemma}[Continuous component]
\label{lem:continuous_component}
If $(x,\theta) \fullyleadsto  (y,\eta)$ then there exist open sets $U$ and $V$,
with $x \in U$ and $y \in V$, and constants $\varepsilon > 0$, $t_0 > 0$, $c >
0$, such that for any $x' \in U$, and all $t \in (t_0, t_0 + \varepsilon]$,
\[
  \prb[x',\theta]{X_t \in \cdot, \Theta_t = \eta} \geq  c \cdot \mathrm{Leb}(\cdot \cap V).
\]
\end{lemma}
\begin{remark} Similar results may be found in previous 
  works, e.g.  \cite[Lemmas 2 and 3]{BakhtinHurth2012}, or \cite[Section
  6.5]{Benaim2015}. In order to get the probabilistic consequences, we need the
  uniformity in the starting point that appears in \cite{Benaim2015}.  Since
  our hypotheses here are slightly different, we include a proof for the sake
  of completeness. 
  We also note that taking canonical switching rates leads to a degenerate
  situation where the local Hörmander type criteria of~\cite{BakhtinHurth2012,Benaim2015} do not apply. 
\end{remark}

\begin{proof}
By hypothesis there exists an admissible deterministic control sequence $\bu = (\bt, \bi) = (t_0,...,t_m;i_1,...,i_m)$, 
such that all indices occur at least once in $\bi$, and $\Phi_\bu(x,\theta) = (y,\eta)$. 
Recall the notation $\tau_k = \sum_{j=0}^{k-1} t_j$ and let  $t= \tau_{m+1} =  \sum_{k=0}^m t_k$
be the final time of the trajectory. 

We use the same thinning construction as in the proof of Lemma~\ref{lem:abs_cont_jumps} above, 
with a Poisson clock of intensity $\overline{\lambda} d$, where $\overline{\lambda}$ 
is an upper bound on the switching rates up to time $t$. 

For $j=1,...,(m-1)$, let $\cU_j$ be a bounded neighbourhood of $\tau_j$; we may assume that 
the $\cU_j$ do not intersect and, by continuity, that the control sequences
$(\bs,\bi) = (s_0,...,s_{m-1},\bi)$ satisfy $\lambda_{\min}(\bs,\bi) \geq \underline{\lambda} >0$ 
for any $\bs$ such that $\sum_{l=0}^{j-1} s_l \in \cU_j$  for all $j$.

Now let $f$ be an arbitrary non-negative test function. Let $A$ be the event
that $m$ Poisson events $T_1,...,T_m$ occur before time $t$, that 
$T_j\in \cU_j$ for all $j$, that the indices are picked as in $\bi$, and that all proposed switches are accepted. Then
\begin{align*}
  \esp{f(X_t,\Theta_t)} 
  &\geq \esp{ f(X_t,\Theta_t) \ind{A}} \\
  &\geq \esp{ f( \Psi(x,t,T_1,...,T_m)) \ind{A}}
\end{align*}
where the mapping $\Psi$ is defined by 
\[ 
  \Psi(x,t,\tau_1,...,\tau_m) = x + \tau_1 \theta + (\tau_2 - \tau_1)F_{i_1}\theta + \cdots 
  + (t- \tau_m) F_{i_1\cdots i_m} \theta. 
\]
Since the choice of indices to switch and the acceptance/rejection tests are
independent from the Poisson process, we get by conditioning:
\begin{align*}
  \esp{f(X_t,\Theta_t)} 
  &\geq \PAR{\frac{\underline{\lambda}}{\overline{\lambda} d}}^m
  \esp{ f( \Psi(x,t,T_1,...,T_m))\ind{\text{$m$ events occur}} \prod_{j=1}^m  \ind{T_j \in \cU_j}}. 
\end{align*}
Using classical properties of the Poisson process, this implies that for some positive
constant $c$, 
\begin{equation}
  \label{eq:reusable_lb}
  \esp{f(X_t,\Theta_t)} 
  \geq c \esp{f(\Psi(x,t,U_1,...,U_m))} 
\end{equation}
where the $U_j$ are independent and $U_j$ is uniformly distributed on $\cU_j$. 

The partial map $(u_1,...,u_m)\mapsto \Psi(x,t,u_1,...,u_m)$ has full rank: 
indeed, the image of its differential is spanned by the vectors 
\[
  (\theta - F_{i_1}\theta, ..., F_{i_1\cdots i_{m-1}}\theta  - F_{i_1\cdots i_m} \theta)
  = (\pm 2 e_{i_1}, ..., \pm 2e_{i_m})
\]
who span $\dR^d$ since all indices in $\{1,...,d\}$  appear at least once 
in the sequence $\bi$. This shows that $\Psi(x,t,\cdot)$ is a submersion.
It follows that, 
$\Psi(x,t,\cdot)$  pushes the uniform distribution on $\prod \cU_j$ to a 
measure which is absolutely continuous with respect to the Lebesgue measure, on 
an open set containing $y = \Phi_\bu(x,\theta)$ (see  \cite[Lemma 2 and 3]{BakhtinHurth2012}, \cite[Section 6]{Benaim2015}
for related results and details). This proves a restricted 
form of the lemma, for the single starting point $x$ and the single time $t$.

To prove the uniform version, we see $x$ and $t$ as a parameter and apply the uniform submersion 
lemma \cite[Lemma 6.3]{Benaim2015} to get the result. 
\end{proof}

\subsection{Non-evanescence}
For classical Markov chains on countable spaces, it is well known
that for any $x$ and $y$, the following equivalence holds:
\[
  \esp[x]{ \sum_n \ind{X_n = y}} = \infty \iff \sum_n\ind{X_n = y} = \infty, \quad \P_x - \text{a.s.}\]
For general chains and processes, this equivalence is no longer true: starting 
from a point $x$, the time spent in a set $A$
may be finite with positive probability, even when its expectation is infinite. This 
may essentially happen if the process has a positive probability of escaping to 
infinity when it starts in a particular set: this canonical counter-example 
is explained e.g. in \cite[Section 9.1.2]{MT09}. 

This equivalence is used to prove that a (classical) irreducible chain that 
admits an invariant probability measure is positive recurrent. To obtain the natural 
property of Harris recurrence for a general chain, ($\psi$-)irreducibility and the 
existence of the invariant probability are not enough, and we need to show additionally 
that the escaping to infinity does not happen. 

In the context of the zigzag process, we refer to the `ridge', Example~\ref{ex:ridge} in Section~\ref{sec:examples}, which describes a smooth potential function with the property that for certain initial conditions the zigzag process will escape to infinity with full probability.

  \begin{definition}[Non-evanescence]
A point $(x,\theta)$ is said to be \emph{non-evanescent} if $\prb[x,\theta]{\abs{X_t}\to \infty} =0$. 
It is \emph{weakly non-evanescent} if this probability is strictly less than $1$. 
\end{definition}

We start by showing how the deterministic statements on flippability may be used 
to prove probabilistic non-evanescence properties. 

\begin{remark}
  [There are infinitely many switches]
  \label{rem:always_switch}
  Note that the first growth condition $ U \to \infty$ already has the 
  probabilistic consequence that the process switches infinitely often. 
 Indeed, for any $(x,\theta)$ and any $n$, 
 \begin{align*}
  \prb[(x,\theta)]{\text{no switch before time $n$}}
  & = \exp\left( -\int_0^{n} \sum_{i=1}^d (\theta_i \partial_i U(x+ \theta s))_+  \, d s\right) \\
  & \leq \exp\left( -\int_0^{n} \sum_{i=1}^d \theta_i \partial_i U(x+ \theta s)  \, d s\right) \\
  & = \exp\left( -U( x + \theta n) + U(x) \right) \rightarrow 0 \quad \text{as $n \rightarrow \infty$}, 
 \end{align*}
 so $\prb[(x,\theta)]{T^1<\infty} = 1$, where $(T^i)$ are the switching times as introduced in Section~\ref{sec:preliminaries}. By the strong Markov property, 
 this implies for all $k$
 \begin{align*}
   \prb[(x,\theta)]{T^{k+1} < \infty} 
   &= \esp[(x,\theta)]{ \ind{T^k<\infty}\prb[(X_{T_k},\Theta_{T^k})]{T^1<\infty}} 
   = \prb[(x,\theta)]{T^k < \infty}, 
 \end{align*}
 proving the claim by recurrence. 
\end{remark}

\begin{lemma}[Two weak versions of non-evanescence]
\label{lem:weakNonEvanescence}
If the invariant measure $\pi$ is a probability measure (as it is assumed to be in this paper), then $\pi$-almost all points are non-evanescent. 

  If additionally the process is fully flippable in the sense of Definition~\ref{def:full_flip},
  then all points are weakly non-evanescent. 
  \end{lemma}
  \begin{proof}
    The first statement is classical.  For the sake of completeness
    we include  a proof. Let $K$ be a compact set. Since $\liminf_{t \rightarrow \infty} \ind{X_t\notin K} = \{ X_t \text{ eventually leaves $K$}\}$, 
    we have by Fatou's lemma
    \[ \prb[\pi]{ X_t \text{ eventually leaves $K$}} \leq \liminf_{t \rightarrow \infty} \prb[\pi]{X_t \notin K} = 1 - \pi(K). \]
    Since $\{ \abs{X_t}\to \infty\} = \bigcap_{K} \{ X_t \text{ eventually leaves $K$}\}$, we are done 
   since $\{\pi\}$ is tight.

   Let us now prove the second statement. Let $\cN$ be the set of non-evanescent points: 
   this set has full $\pi$-measure, so its complement is Lebesgue negligible. 
   Let $(x,\theta)$ be an arbitrary starting point, 
   and consider the stopping time $\tau$ introduced in Lemma~\ref{lem:abs_cont_jumps}. 
   By the strong Markov property,
   \begin{align*}
     \prb[(x,\theta)]{ \abs{X_t} \text{ does not go to infinity}}
     &\geq
     \prb[(x,\theta)]{ \tau < \infty, \abs{X_t} \text{ does not go to infinity}} \\
     &= \esp[(x,\theta)]{ \ind{\tau<\infty} \prb[(X_\tau,\Theta_\tau)]{ \abs{X_t} \text{ does not go to infinity}}} \\
     &\geq \esp[(x,\theta)]{ \ind{\tau<\infty} \ind{X_\tau \in \cN}}. 
   \end{align*}
   Since $\dR^d\setminus \cN$ is Lebesgue negligible, $\prb[(x,\theta)]{ \tau<\infty, X_\tau \notin \cN} = 0$, so  
   \begin{equation}
     \label{eq:lower_bound_no_escape}
     \prb[(x,\theta)]{ \abs{X_t} \text{ does not go to infinity}} \geq \prb[(x,\theta)]{\tau<\infty}.
   \end{equation}
   If the process is fully flippable, this last probability is positive, proving the
   weak non-evanescence property. 
  \end{proof}

  If we add a slightly stronger hypothesis on the growth of the potential at infinity, namely Growth Condition~\ref{ass:GC2},
  we get a stronger non-evanescence result. We start by saying that if the process
  is evanescent, it must go to infinity in a very particular way, by staying forever in 
  an affine subspace.
\begin{lemma}[Two frozen directions]
  \label{lem:frozen}
  Let $d\geq 2$.  
  Suppose that there exists an invariant probability measure, and 
 that $(x,\theta)$ satisfies
  $\prb[(x,\theta)]{\abs{X_t} \to \infty} >0$. Then 
  there exist two indices $i$ and $j$ such that 
  \[ \prb[(x,\theta)]{ \text{the $i$\textsuperscript{th} and $j$\textsuperscript{th}
	components never switch}} >0.\]
\end{lemma}
\begin{proof}
  We prove this statement by contraposition and assume that, with probability one, 
  at most one component of the velocity does not switch. This implies 
  that the time $T_N$ defined in Lemma~\ref{lem:abs_cont_jumps} is a.s. finite, 
  and since there are infinitely many switches by Remark~\ref{rem:always_switch}, 
  the time $\tau=T_{N+1}$ of the same Lemma~\ref{lem:abs_cont_jumps} is also finite. 
  Reusing the bound~\eqref{eq:lower_bound_no_escape} from the proof of Lemma~\ref{lem:weakNonEvanescence}, 
  we immediately get that $\prb[(x,\theta)]{\abs{X_t}\to \infty} = 0$, proving the lemma.
\end{proof}

  Recall that Growth Condition~\ref{ass:GC2} states, in dimension $d$, that 
  \[ \exists c> d, \exists c', \forall x, \quad U(x) \geq c\ln(1+\abs{x}) - c'.\]
\begin{proposition}[Non-evanescence]
  \label{prop:SNE}
  If the potential $U$ satisfies Growth Condition~\ref{ass:GC2} 
  then  the process is non-evanescent, 
  that is, for any $(x,\theta)\in\dR^d\times \{-1,1\}^d$, 
  \[ \prb[(x,\theta)]{ \abs{X_t} \to \infty} = 0. \]
\end{proposition}

\begin{proof}[Proof of Proposition~\ref{prop:SNE}]

We wish to prove for all $d$ the following statement:
  \begin{equation} \label{eq:nonevanescence-induction}\tag{$\mathcal{P}_d$}
     \forall U:\dR^d\to \dR,\quad  \text{$U$ satisfies GC2}
   \implies \text{the zigzag process for $U$ is non-evanescent}.
 \end{equation}
 
If $d = 1$, by~\eqref{eq:lower_bound_no_escape}, with $\tau$ denoting the time of the first switch, and Remark~\ref{rem:always_switch}, \eqref{eq:nonevanescence-induction} follows. 

For $d \geq 2$, the strategy is to prove this by induction. The form of the growth condition
is tailored to this strategy: it clearly implies that $\int \exp(-U(x)) dx$ is finite and may  be normalized into a probability, but it crucially also implies that the same is true for all the conditional measures on affine subspaces.
For the base case $d=2$, 
  using Lemma~\ref{lem:frozen}, we see that if $\prb[(x,\theta)]{|X_t|\to \infty} >0$ 
  then with positive probability the process never switches.
  Since $U\to \infty$ this is not possible (see Remark~\ref{rem:always_switch}).

  Let us now prove the induction step by contraposition. Assume that 
  ($\mathcal{P}_{d+1}$) is false: there exists a potential $U$ in dimension $d+1$ that 
 satisfies the growth condition, but for which the zigzag process is evanescent, that
is,  there is
  a point $(x,\theta)$ such that $\prb[(x,\theta)]{\abs{X_t} \to \infty} >0$.
  Our goal is to define a potential in dimension $d$ that also satisfies the growth condition and for which we also have evanescence.

  By Lemma~\ref{lem:frozen}, there are two indices, say $d$ and $d+1$ without loss of
  generality, such that 
  \[
    \prb[(x,\theta)]{ \text{$d$ and $d+1$ never switch}} >0.
  \]
  We may also assume without loss of generality that $\theta_d = \theta_{d+1} =
  1$. Note that the process may be constructed by considering $d+1$ 
  sequences of iid exponential random variables $(E^k_j)_{j=1,...,d+1; k\in\dN}$ and saying
  that the $k$\textsuperscript{th} jump of the $j$\textsuperscript{th} component
  of $\Theta$, say $T^k_j$,  occurs when the accumulated jump rate $\int_{T^{k-1}_j}^t \lambda_j(X_s,\Theta_s) ds$ 
  reaches $E^k_j$.

  Consider now a second, $d$-dimensional zigzag process $(Y_1,...,Y_d;H_1,...,H_d)$
  starting from $(x_1,...,x_d;\theta_1,...,\theta_d)$ in the potential
  $V(y_1,...,y_d) = U(y_1,...,y_d,y_d)$. 
  Note that, since $U$ satisfies the growth condition, 
  \[ V(y_1,...,y_d) \geq c\ln( 1+ \abs{(y_1,...,y_d,y_d)}_{\dR^{d+1}}) - c'
              \geq  c\ln( 1+ \abs{(y_1,...,y_d)}_{\dR^{d}}) - c'
	    \]
	    where $c>d+1>d$, so $V$ satisfies the growth condition
	in dimension $d$. It remains to show that 
	the zigzag process in $V$ is evanescent. 

  We couple the process in $V$ with the previous one,  using the same randomness 
  $(E^k_j)_{j=1,\dots,d-1, k\in\dN}$ for the first $d-1$ coordinates, and an 
  independent sequence $(\tilde{E}^k_d)_{k\in\dN}$ for the last one. 
  Let $\tau$ be the first time when one of $\Theta_{d}$, $\Theta_{d+1}$ or $H_d$
  switches. 
  For $t\leq \tau$, using the elementary bound $(a+b)_+ \leq a_+ + b_+$ and the 
 fact that  $H_d$, $\Theta_d$ and $\Theta_{d+1}$ are all equal to $1$ up to time $t$, 
 we get
  \begin{align*}
    \int_0^t (\partial_d V(Y_s) H_d(s))_+ ds
    &= \int_0^t (\partial_d U(Y_s) + \partial_{d+1}U(Y_s))_+ ds \\
    &= \int_0^t (\partial_d U(X_s) + \partial_{d+1}U(X_s))_+ ds \\
    &\leq \int_0^t (\Theta_d(s) \partial_d U(X_s))_+ ds + \int_0^t (\Theta_{d+1}(s) \partial_{d+1}U(X_s))_+ ds \\
    &\leq \int_0^\infty (\Theta_d(s) \partial_d U(X_s))_+ ds + \int_0^\infty (\Theta_{d+1}(s) \partial_{d+1}U(X_s))_+ ds. 
    \end{align*}
  Now, the event $A = \{ \tilde{E}^1_d \geq E^1_d + E^1_{d+1} \} \cap \{ \Theta_d \text{ and } \Theta_{d+1} \text{ never switch}\}$
  has positive probability, and on this event we can continue the bounds: 
  \begin{align*}
    \int_0^t (\partial_d V(Y_s) H_d(s))_+ ds
    &\leq \int_0^\infty (\Theta_d(s) \partial_d U(X_s))_+ ds + \int_0^\infty (\Theta_{d+1}(s) \partial_{d+1}U(X_s))_+ ds \\
    &< E^1_d + E^1_{d+1} \\
    &\leq \tilde{E}^1_d. 
    \end{align*}
    This shows that on $A$, $\tau$ must be infinite, that is, $H_d$ never switches either and thus $|Y_t| \rightarrow \infty$.
    Since the growth hypothesis is satisfied for $V$, 
    this concludes the proof of the induction step by contraposition. 
\end{proof}

\subsection{Putting the pieces together}
\begin{theorem}
  \label{thm:Tprocess}
 If the zigzag process is fully flippable, then it is a weakly non-evanescent $T$-process. 

 If in addition $(x,\theta)\leadsto (y,\eta)$ for all pairs of points, the process
 is $\psi$-irreducible and aperiodic, and all compact sets are petite.

 If in addition the process is (strongly) non-evanescent, then 
 it is positive Harris recurrent and ergodic. 
\end{theorem}

\begin{proof}
The fact that a fully flippable zigzag process is weakly non-evanescent is a consequence of Lemma~\ref{lem:weakNonEvanescence}.

 We know that all points $(x,\theta)\in E$ lead
  to a different point by a sequence where all indices are switched.  From
  Lemma~\ref{lem:continuous_component} and a  compactness argument, this
  implies that there exists a family $(\cU_n)_{n\in\dN}$ of open sets in $E$, a
  family $(\cV_n)_{n\in\dN}$ of open sets in $\dR^d$, velocities
  $\eta_n\in\{-1,1\}^d$  and numbers $(t_n, \eps_n, c_n)$, such that:
 \begin{itemize}
   \item The $(\cU_n)_{n\in\dN}$ form a locally finite open cover: each $(x,\theta)\in E$ 
    belongs to at least one, and at most a finite number of the $\cU_n$. 
  \item  for all 
 $(x,\theta)\in \cU_n$, all $t\in[t_n,t_n+\eps_n]$ and all positive measurable $f$, 
 \[ \esp[(x,\theta)]{ f(X_t,\Theta_t)} 
   \geq c_n \int f(y,\eta_n) \ind{\cV_n}(y) dy. \] 
 \end{itemize}
 Define a kernel $K$ by the formula
 \[ 
   K((x,\theta), A \times \{\eta\}) 
   = \int \ind{A}(y) \max_{n: (x,\theta) \in \cU_n}\PAR{%
     c_n \ind{\eta_n = \eta} \ind{\cV_n}(y) \int_{t_n}^{t_n+ \eps_n} e^{-t}dt
   }  dy.
 \]
 By construction, the resolvent is bounded below by $K$. 
 For all $(x,\theta) \in \cU_n$, we have that
 $K((x,\theta),E) \geq c_n \text{Leb}(\cV_n)\int_{t_n}^{t_n+\eps_n}e^{-t}dt>0$, i.e. $K$ is nontrivial. 
 Moreover, for any measurable set $A$ and any $\eta$, $K((x,\theta),A\times\{\eta\})$ is lower semicontinuous in $(x,\theta)$: indeed, if $(x_j)$ converges to $x$, then the
 $x_j$ will eventually belong to all the $\cU_n$ containing $x$, so
 $K((x_j,\theta),A) \geq K((x,\theta),A)$ for $j$ large enough.  
 To sum up, the resolvent kernel of the process is bounded below 
 by a nontrivial lower semi continuous kernel: the process is a $T$-process. 

 Suppose now that  $(x,\theta)\leadsto (y,\eta)$ for all pairs of points. This implies
that $(x,\theta)\fullyleadsto (y,\eta)$ for all pairs of points. For any 
such pair, and any neighbourhood $\cO\times\{\eta\}$ of $(y,\eta)$, another application
of Lemma~\ref{lem:continuous_component} yields $\prb[x,\theta]{\tau_\cO<\infty} >0$;
this in turn implies that 
the process is open set irreducible in the sense of \cite{Twe94}. 
By \cite[Theorem 3.2]{Twe94} (see also \cite[Proposition 6.2.2]{MT09} for 
the similar statement for discrete time chains), the process is then $\psi$-irreducible. 

All compact sets are petite by an application~\cite[Theorem 4.1 (i)]{MT2}.

To prove aperiodicity, 
let $(x,\theta)$ be an arbitrary point. We know that $(x,\theta)\fullyleadsto (x,\theta)$, 
  so by Lemma~\ref{lem:continuous_component}, there exists $t_0,\eps$ and two 
  open neighbourhoods $\cU$ and $\cV$  of $x$  such that
  \begin{equation}
    \label{eq:UandV}
    \prb[x',\theta]{X_t \in \cdot, \Theta_t = \theta} \geq c \mathrm{Leb}(\cdot\cap \cV),
  \end{equation}
  for all $x'\in \cU$ and $t\in[t_0,t_0+\eps]$. This shows
  that $\cU$ is a petite set.
    Writing  $\cW  =\cU\cap \cV$,  we see that $\cW$ is petite (as a subset of $\cU$), 
  and for all $x'\in \cW $ and $t\in [t_0,t_0+\eps]$, 
  \[ 
    \prb[x',\theta]{X_t \in \cW , \Theta_t = \theta} \geq c', 
  \]
  where $c'= c \mathrm{Leb}(\cW)$.
  Let $N = \lceil t_0/\eps\rceil$ and $T=Nt_0$. 
  For any $t\geq T$, let $n = \lfloor t/t_0 \rfloor$
  and $t_0' = t/n$. Then $t_0'\in[t_0,t_0+\eps]$, so by iteration and the Markov property, 
  \[ \prb[x',\theta]{X_t \in \cW , \Theta_t = \theta} \geq (c')^n >0,\]
  proving the aperiodicity.

To prove Harris recurrence, we use the fact that for $\psi$-irreducible $T$-processes, 
it is in fact equivalent to non-evanescence (\cite[Theorem 3.2]{MT2}), 
and the positivity follows from the fact that there is an invariant probability 
measure. 

It remains to show that the process is ergodic. By \cite[Theorem 6.1]{MT2}, 
it is enough to prove that some skeleton chain is irreducible. To this end, 
first take $(x,\theta)$ an arbitrary point: we reuse Lemma~\ref{lem:continuous_component}
to define $\cU$, $\cV$, $t_0$ and $\eps$ such that Eq.~\eqref{eq:UandV} holds; 
in words, it is possible to loop around $(x,\theta)$ and there is a little room $\eps$ 
in the looping time. 
Now let $(y,\eta)$, $(y',\eta')$ be two arbitrary points. By reachability 
we
can go from the first one to the second one with a visit to $(x,\theta)$ in between, 
and adding a loop around $(x,\theta)$ in the middle will give us what we need. 
More formally, using Lemma~\ref{lem:continuous_component} twice more, 
there exists $t_1$, $c_1$ and a neighborhood  $\cV_1$ of $x$ such that
\[
  \prb[(y,\eta)]{ (X_{t_1},\Theta_{t_1}) \in \cdot\times\{\theta\}}  \geq c_1 \mathrm{Leb}(\cdot\cap \cV_1),
\]
and $t_2$, $c_2$ and two neighborhoods $\cU_2$ and $\cV_2$ of $x$ and $y'$ such that
\[
  \prb[(x',\theta)]{ (X_{t_2},\Theta_{t_2}) \in \cdot \times \{\eta'\}} 
  \geq c_2\mathrm{Leb}(\cdot\cap \cV_2)
\]
for all $x'\in \cU_2$. 
Then for any $t \in [t_0+t_1+t_2, t_0+t_1+t_2+\eps]$, 
applying the Markov property at the times $t_1$ and $t-t_2$ yields
\begin{align*}
  & \prb[(y,\eta)] {(X_t,\Theta_t) \in \mathcal O \times \{\eta'\}} \\
  & \geq\prb[(y,\eta)] {\Theta_{t_1}  = \Theta_{t-t_2} = \theta,\Theta_t = \eta', X_{t_1}\in \cU\cap\cV_1, X_{t-t_2}\in \cV\cap\cU_2,X_t\in \cO} \\
  & \geq\prb[(y,\eta)] {\Theta_{t_1}  = \Theta_{t-t_2} = \theta, X_{t_1}\in \cU\cap\cV_1, X_{t-t_2}\in \cV\cap\cU_2}c_2\mathrm{Leb}(\cO\cap\cV_2) \\
  & \geq\prb[(y,\eta)] {\Theta_{t_1}  = \theta, X_{t_1}\in \cU\cap\cV_1} c\mathrm{Leb}(\cV\cap\cU_2) c_2\mathrm{Leb}(\cO\cap\cV_2) \\
  &\geq c c_1 c_2 \mathrm{Leb}(\cU\cap\cV_1)\mathrm{Leb}(\cV\cap\cU_2) \mathrm{Leb}(\cO\cap\cV_2), 
\end{align*}
since $(t-t_2) - t_1 \in [t_0,t_0+\eps]$. The time interval
 $[t_0+t_1+t_2, t_0+t_1+t_2 +\eps]$ must contain a multiple of $\eps$, proving that the $\eps$-chain is open
set irreducible and therefore irreducible. 
\end{proof}

\subsection{Lyapunov function}
\label{sec:Lyapunov}

In order to establish exponential ergodicity we have to establish contractivity in the tails for which a Lyapunov function argument is used. For this we first require the notion of the generator of the zigzag process.
We define the \emph{generator} of the zigzag process in $E$ with switching rates $(\lambda_i)_{i=1}^d$ as the operator $L$ whose domain $\mathcal D(L)$ consists of continuous functions $f : E \rightarrow \R$, such that $t \mapsto f(x+\theta t,\theta)$ is absolutely continuous on $[0,\infty)$ for all $(x,\theta) \in E$. For such $f \in \mathcal D(L)$, the function $Lf$ is defined as 
\[ L f(x,\theta) = \langle \theta, \nabla f(x) \rangle + \sum_{i=1}^d \lambda_i(x,\theta) (f(x, F_i \theta) - f(x,\theta)), \quad  (x,\theta) \in E.\]

The main result on exponential ergodicity (Theorem~\ref{thm:exponential-ergodicity}) will be proved using the following result from Down, Meyn and Tweedie (\cite[Theorem 5.2]{DMT95}). 
\begin{theorem}[Drift criterion for exponential convergence]
\label{thm:downmeyntweedie}
  Suppose that $(X_t,\Theta_t)$ is an irreducible aperiodic process, 
  and suppose that there exists a Lyapunov function, that is, a function $V\geq 1$ such that 
  \[ LV \leq - \eps V + c\ind{K}, \]
  where $K$ is a petite set. Then $(X_t,\Theta_t)$ is exponentially ergodic: 
  \[ \| \prb[(x,\theta)]{ (X_t,\Theta_t) \in \cdot} - \pi \|_{\mathrm{TV}} \leq M(x,\theta)e^{-ct},\]
  for some positive constant $c$. 
\end{theorem}

As discussed in \cite{DMT95}, the function $M(x,\theta)$ may be taken to be a positive multiple of $V$. The approach in \cite{DMT95} does not yield quantitative results on the value of $c$. For estimates on the rate of convergence in an $L^2$-framework of the Zig-Zag processes (and other piecewise deterministic process) we refer to \cite{Andrieu2018}.

\begin{remark} 
  The continuity assumption on functions in the domain $\mathcal D(L)$ leads to
  a domain which is somewhat smaller than that of the \emph{extended
    generator}, characterized in \cite[Theorem 26.14]{Davis1993}. However this
  definition is sufficient for our purposes. 
\end{remark}
In order to motivate our choice of Lyapunov function, first note that
we are looking for a function that typically decreases along the dynamics.
Since the velocity has a positive probability of switching whenever the process
is going "uphill" (that is, whenever $\scal{\theta,\nabla U(x)}>0$, 
 a first guess might be $V(x,\theta) = \exp(\alpha U(x))$ for some $\alpha > 0$.
 However this velocity jump will not occur immediately, 
therefore we wish to introduce a dependence on the partial derivatives of $U$
and on the direction $\theta$ so that the effect of the switching intensity is
to decrease $V$ with sufficiently large probability   while we are running
uphill of the potential. For a zero excess switching rate, $\gamma(x,\theta)
\equiv 0$, we could simply take $V(x,\theta) = \exp(\alpha U(x) +\beta \langle
\theta, \nabla U(x)\rangle)$ but for nonzero excess switching rate we have to
be more careful in dependence on the partial derivatives of $U$. The particular
structure of the zigzag process enables us to work on each component of the
gradient separately.

The Lyapunov function used for the one-dimensional zigzag process (see \cite{BierkensRoberts2015}) requires milder assumptions compared to Growth Condition~\ref{ass:GC3}: it only requires $|U'(x)|$ to be bounded away from zero for $x$ outside of a compact set, without any conditions on the second derivative. However, it cannot be extended to the multi-dimensional case in a simple way. Indeed, the multi-dimensional generalization
\[ V(x,\theta) = \exp\left( \alpha \|x\| + \beta \langle \theta, x/\|x\| \rangle \right)\]
fails to be contractive in e.g. the case of a non-diagonally dominant Gaussian target.

The Lyapunov function we will introduce in Lemma~\ref{lem:lyapunov} may also be
compared to the Lyapunov function for the Bouncy Particle Sampler
\cite{Deligiannidis2017},
\[ V(x,v) =  \exp\left(
    \tfrac 1 2 U(x)) - \tfrac 1 2 \ln(\lambda(x,-v)
  \right), \quad (x,v) \in \R^d \times S^{d-1}.
\] 
Note
that this Lyapunov function is not well defined in our situation which should
include the case of canonical switching rates, where $\gamma(\cdot) \equiv 0$. 

\begin{lemma}
\label{lem:lyapunov}
  Suppose Growth Condition~\ref{ass:GC3} is satisfied. Consider the process with a switching rate given by $\lambda_i(x,\theta) = \gamma_i(x,\theta) + (\theta_i \partial_i U(x))_+$, where $\gamma : E \rightarrow [0,\infty)^d$ is bounded: for some constant $\overline \gamma \geq 0$,
  \[ \gamma_i(x,\theta) \leq \overline \gamma, \quad (x,\theta) \in E, \, i = 1, \dots, d.\]
  Let $\delta > 0$ and $\alpha > 0$ such that $0 \leq \overline \gamma \delta  <\alpha < 1$. Define $\phi(s) = \tfrac 1 2 \operatorname{sign}(s)
   \ln\PAR{ 1+ \delta \abs{s}}$. Then the function 
  \begin{equation}\label{eq:lyapunov} V(x,\theta) = \exp\PAR{  \alpha U(x) + \sum_i \phi(\theta_i \partial_i U(x))}\end{equation}
  is a Lyapunov function for $(X_t,\Theta_t)$, that is, $\lim_{|x| \rightarrow \infty} V(x) = \infty$ and 
  \[ LV \leq - \eps V + C \ind{K},\]
  where $\eps$, $C$ are positive constants and $K$ is a compact set in $E$. 
\end{lemma}
\begin{proof} It may be verified that $V \in \mathcal D(L)$.
   Using the expression of the generator, 
   \begin{align*}
     (LV/V)(x,\theta) &= \alpha \scal{\theta,\nabla U(x)} 
     + \sum_{i,j} \theta_i \partial_{ij}U(x) \theta_j \phi'(\theta_j \partial_j U(x)) \\
     &\quad + \sum_i (\gamma_i + (\theta_i \partial_i U)_+) \PAR{ \exp( \phi( - \theta_i \partial_i U) - \phi(\theta_i\partial_iU)) - 1}
   \end{align*}
For the $i^{\text{th}}$ component, if $s = \theta_i \partial_i U \geq 0$, then $\phi(-s) - \phi(s) = - \ln(1+\delta s)$, so
\begin{align*}  & \alpha s + (\gamma_i + (s)_+)\PAR{ \exp( \phi( - s) - \phi(s)) - 1} \\
 & = (\alpha - 1) s + \frac{(1 - \delta \gamma_i)s}{ 1 + \delta s} \leq -(1-\alpha)|s| + (1/\delta).
\end{align*}
   When $s<0$, we have $\phi(-s)-\phi(s) = \ln(1+ \delta \abs{s})$, so
   \begin{align*}
    & \alpha s + (\gamma_i + (s)_+)\PAR{ \exp( \phi( - s) - \phi(s)) - 1} \\
    & = \alpha s + \gamma_i \left( 1 + \delta \abs {s} - 1 \right) \leq - (\alpha - \overline \gamma \delta) |s|. 
   \end{align*}
   In either case, 
   \[ \alpha s + (\gamma_i + (s)_+)\PAR{ \exp( \phi( - s) - \phi(s)) - 1} \leq -  \min(1-\alpha, \alpha - \delta \overline \gamma) \abs{s} + (1/\delta).\] Since $0\leq \phi'(s)\leq \delta/2$, 
   \begin{align*}
     (LV/V)(x,\theta)   
       &\leq - \min(1-\alpha, \alpha - \overline \gamma \delta) \sum_i \abs{\partial_i U} 
       +d /\delta
     + \frac{\delta}{2} \sum_{i,j} \abs{ \partial_{ij} U}, 
   \end{align*}
   which is less than $1$ outside a sufficiently large ball by our hypotheses. 
 \end{proof}

\subsection{Proofs of the main results}
\label{sec:mainresults-proofs}

\begin{proof}[Proof of Theorem~\ref{thm:ergodicity}]
The steps of the proof are completely as depicted in Figure~\ref{fig:diagram} and simply consist of combining Proposition~\ref{prop:full_flip}, Theorem~\ref{thm:reachability} and Theorem~\ref{thm:Tprocess}. 
\end{proof}

\begin{proof}[Proof of Theorem~\ref{thm:exponential-ergodicity}]
By Lemma~\ref{lem:lyapunov}, there exists a Lyapunov function $V$ such that for some $\varepsilon > 0$, $LV \leq -\varepsilon V$ outside a compact set, where $L$ is the generator of the zigzag process, see Section~\ref{sec:Lyapunov}. Since Growth Condition~\ref{ass:GC3} implies Growth Condition~\ref{ass:GC1}, by Theorem~\ref{thm:Tprocess}, all compact sets are petite, and the process is $\psi$-irreducible and aperiodic, so that the conditions of Theorem~\ref{thm:downmeyntweedie} are satisfied, which establishes exponential ergodicity.
\end{proof}

\begin{proof}[Proof of Theorem~\ref{thm:FCLT}]
By the growth condition, there exist $\alpha > 0$ such that $\alpha < \beta + \eta/4 < 1/2$ and $\delta > 0$ such that $0 < \delta < \alpha$ such that, for some $c > 0$, $g \leq c V$ with $V$ given by~\eqref{eq:lyapunov}. 
Furthermore, again by the growth condition, for $x$ outside a bounded set, $V(x,\theta) \leq \exp( (\beta + \eta/2) U(x))$.
From the integrability assumption, $\pi(V^2) < \infty$. That all compact sets are petite follows from Theorem~\ref{thm:Tprocess}, whose conditions are satisfied by Theorem~\ref{thm:reachability}. The statement of the theorem then follows from Lemma~\ref{lem:lyapunov} and \cite[Theorem 4.3]{GlynnMeyn1996}.
\end{proof}

\section*{Acknowledgements}

We thank Tony Lelièvre, Paul Fearnhead and Eva Löcherbach for stimulating discussions, Pierre Monmarché for many 
exchanges on the merits of various Lyapunov functions, and Nikolas Nuesken and Julien Roussel
for discussions on alternative approaches. We thank the associate editor and the anonymous referee for their comments which helped to correct and improve this manuscript.


\begin{thebibliography}{10}
	
	\bibitem{Andrieu2018}
	Christophe Andrieu, Alain Durmus, Nikolas N{\"{u}}sken, and Julien Roussel.
	\newblock {Hypercoercivity of Piecewise Deterministic Markov Process-Monte
		Carlo}.
	\newblock {\em arXiv preprint arXiv: 1808.08592}, aug 2018.
	
	\bibitem{ABGKZ12}
	Romain Azaïs, Jean-Baptiste Bardet, Alexandre G\'enadot, Nathalie Krell, and
	Pierre-Andr\'e Zitt.
	\newblock Piecewise deterministic {M}arkov process---recent results.
	\newblock In {\em Journ\'ees {MAS} 2012}, volume~44 of {\em ESAIM Proc.}, pages
	276--290. EDP Sci., Les Ulis, 2014.
	
	\bibitem{BakhtinHurth2012}
	Y.~Bakhtin and T.~Hurth.
	\newblock {Invariant densities for dynamical systems with random switching}.
	\newblock {\em Nonlinearity}, 25(10):2937--2952, 2012.
	
	\bibitem{Benaim2015}
	M.~Benaim, S.~{Le Borgne}, F.~Malrieu, and P.-A. Zitt.
	\newblock {Qualitative properties of certain piecewise deterministic Markov
		processes}.
	\newblock {\em Ann. Inst. Henri Poincar{\'{e}} Probab. Stat.},
	51(3):1040--1075, 2015.
	
	\bibitem{BierkensDuncan2016}
	J.~Bierkens and A.~Duncan.
	\newblock {Limit theorems for the Zig-Zag process}.
	\newblock {\em Advances in Applied Probability}, 49(3), jul 2017.
	
	\bibitem{BierkensFearnheadRoberts2016}
	J.~Bierkens, P.~Fearnhead, and G.~O. Roberts.
	\newblock {The Zig-Zag Process and Super-Efficient Sampling for Bayesian
		Analysis of Big Data}.
	\newblock {\em to appear in Annals of Statistics}, 2018.
	
	\bibitem{Bierkens2015}
	Joris Bierkens.
	\newblock {Non-reversible Metropolis-Hastings}.
	\newblock {\em Statistics and Computing}, 25:1--16, 2015.
	
	\bibitem{Bierkens2018a}
	Joris Bierkens, Kengo Kamatani, and Gareth~O. Roberts.
	\newblock {High-dimensional scaling limits of piecewise deterministic sampling
		algorithms}.
	\newblock {\em arXiv preprint arXiv: 1807.11358}, jul 2018.
	
	\bibitem{BierkensRoberts2015}
	Joris Bierkens and Gareth Roberts.
	\newblock {A piecewise deterministic scaling limit of lifted
		Metropolis--Hastings in the Curie--Weiss model}.
	\newblock {\em Ann. Appl. Probab.}, 27(2):846--882, 2017.
	
	\bibitem{BouchardCoteVollmerDoucet2017}
	Alexandre Bouchard-C{\^{o}}t{\'{e}}, Sebastian~J Vollmer, and Arnaud Doucet.
	\newblock {The Bouncy Particle Sampler: A Non-Reversible Rejection-Free Markov
		Chain Monte Carlo Method}.
	\newblock {\em Journal of the American Statistical Association}, 2017.
	
	\bibitem{Davis1993}
	M.~H.~A. Davis.
	\newblock {\em {Markov models and optimization}}, volume~49 of {\em Monographs
		on Statistics and Applied Probability}.
	\newblock Chapman {\&} Hall, London, 1993.
	
	\bibitem{Deligiannidis2017}
	George Deligiannidis, Alexandre Bouchard-C{\^{o}}t{\'{e}}, and Arnaud Doucet.
	\newblock {Exponential Ergodicity of the Bouncy Particle Sampler}.
	\newblock {\em arXiv preprint arXiv: 1705.04579}, 2017.
	
	\bibitem{Deligiannidis2018}
	George Deligiannidis, Daniel Paulin, and Arnaud Doucet.
	\newblock {Randomized Hamiltonian Monte Carlo as Scaling Limit of the Bouncy
		Particle Sampler and Dimension-Free Convergence Rates}.
	\newblock {\em arXiv preprint arXiv: 1808.04299}, aug 2018.
	
	\bibitem{DiaconisHolmesNeal2000}
	Persi Diaconis, Susan Holmes, and RM~Neal.
	\newblock {Analysis of a nonreversible Markov chain sampler}.
	\newblock {\em Annals of Applied Probability}, 10(3):726--752, 2000.
	
	\bibitem{DMT95}
	D.~Down, S.~P. Meyn, and R.~L. Tweedie.
	\newblock Exponential and uniform ergodicity of {M}arkov processes.
	\newblock {\em Ann. Probab.}, 23(4):1671--1691, 1995.
	
	\bibitem{DuncanLelievrePavliotis2015}
	A.~B. Duncan, T.~Leli{\`{e}}vre, and G.~A. Pavliotis.
	\newblock {Variance Reduction using Nonreversible Langevin Samplers}.
	\newblock {\em Journal of Statistical Physics}, 163(3):457--491, jun 2016.
	
	\bibitem{Durmus2018}
	Alain Durmus, Arnaud Guillin, and Pierre Monmarch{\'{e}}.
	\newblock {Geometric ergodicity of the bouncy particle sampler}.
	\newblock {\em arXiv preprint arXiv: 1807.05401}, jul 2018.
	
	\bibitem{Fet17}
	Ninon Fetique.
	\newblock Long-time behaviour of generalised zig-zag process, 2017.
	
	\bibitem{GlynnMeyn1996}
	Peter~W. Glynn and Sean~P. Meyn.
	\newblock {A Liapounov bound for solutions of the poisson equation}.
	\newblock {\em Annals of Probability}, 24(2):916--931, 1996.
	
	\bibitem{Hwang1993}
	CR~Hwang, SY~Hwang-Ma, and SJ~Sheu.
	\newblock {Accelerating Gaussian diffusions}.
	\newblock {\em The Annals of Applied Probability}, 3(3):897--913, 1993.
	
	\bibitem{johnson2012variable}
	Leif~T Johnson and Charles~J Geyer.
	\newblock {Variable transformation to obtain geometric ergodicity in the
		random-walk Metropolis algorithm}.
	\newblock {\em The Annals of Statistics}, pages 3050--3076, 2012.
	
	\bibitem{Lelievre2013}
	T.~Leli{\`{e}}vre, F.~Nier, and G.~A. Pavliotis.
	\newblock {Optimal Non-reversible Linear Drift for the Convergence to
		Equilibrium of a Diffusion}.
	\newblock {\em Journal of Statistical Physics}, 152(2):237--274, jun 2013.
	
	\bibitem{LPW09}
	D.~A. Levin, Y.~Peres, and E.~L. Wilmer.
	\newblock {\em {Markov chains and mixing times}}.
	\newblock American Mathematical Society, 2009.
	
	\bibitem{Ma2016}
	Y.-A. Ma, T.~Chen, L.~Wu, and E.~B. Fox.
	\newblock {A Unifying Framework for Devising Efficient and Irreversible MCMC
		Samplers}, 2016.
	
	\bibitem{Mal15}
	Florent Malrieu.
	\newblock Some simple but challenging {M}arkov processes.
	\newblock {\em Ann. Fac. Sci. Toulouse Math. (6)}, 24(4):857--883, 2015.
	
	\bibitem{Maruyama1959}
	Gisiro Maruyama and Hiroshi Tanaka.
	\newblock {Ergodic Property of N-Dimensional Recurrent Markov Processes}.
	\newblock {\em Memoirs of the Faculty of Science, Kyushi University, Series A},
	13(2):157--172, 1959.
	
	\bibitem{Metropolis1953}
	Nicholas Metropolis, Arianna~W. Rosenbluth, Marshall~N. Rosenbluth, Augusta~H.
	Teller, and Edward Teller.
	\newblock {Equation of State Calculations by Fast Computing Machines}.
	\newblock {\em The Journal of Chemical Physics}, 21(6):1087, 1953.
	
	\bibitem{MeynTweedie1992}
	S.~Meyn and R.~L. Tweedie.
	\newblock {Stability of Markovian Processes I: Criteria for Discrete-Time
		Chains}.
	\newblock {\em Advances in Applied Probability}, 24(3):542--574, 1992.
	
	\bibitem{MT3}
	S.~Meyn and R.~L. Tweedie.
	\newblock {Stability of Markovian processes III: Foster-Lyapunov criteria for
		continuous-time processes}.
	\newblock {\em Advances in Applied Probability}, 25(3):518--548, 1993.
	
	\bibitem{MT09}
	Sean Meyn and Richard~L. Tweedie.
	\newblock {\em Markov chains and stochastic stability}.
	\newblock Cambridge University Press, Cambridge, second edition, 2009.
	\newblock With a prologue by Peter W. Glynn.
	
	\bibitem{MT2}
	Sean~P. Meyn and R.~L. Tweedie.
	\newblock Stability of {M}arkovian processes. {II}. {C}ontinuous-time processes
	and sampled chains.
	\newblock {\em Adv. in Appl. Probab.}, 25(3):487--517, 1993.
	
	\bibitem{MichelKapferKrauth2014}
	M.~Michel, S.~C. Kapfer, and W.~Krauth.
	\newblock {Generalized event-chain Monte Carlo: Constructing rejection-free
		global-balance algorithms from infinitesimal steps}.
	\newblock {\em The Journal of Chemical Physics}, 140(5), 2014.
	
	\bibitem{Monmarche2016}
	P.~Monmarch{\'{e}}.
	\newblock {Piecewise deterministic simulated annealing}.
	\newblock {\em ALEA}, 13(1):357--398, 2016.
	
	\bibitem{Pakman2017}
	Ari Pakman.
	\newblock {Binary Bouncy Particle Sampler}.
	\newblock {\em arXiv preprint arXiv: 1711.00922}, 2017.
	
	\bibitem{Pakman2016}
	Ari Pakman, Dar Gilboa, David Carlson, and Liam Paninski.
	\newblock {Stochastic Bouncy Particle Sampler}, 2016.
	
	\bibitem{PetersDeWith2012}
	E.~A. J.~F. Peters and G.~{De With}.
	\newblock {Rejection-free Monte Carlo sampling for general potentials}.
	\newblock {\em Physical Review E - Statistical, Nonlinear, and Soft Matter
		Physics}, 85(2):1--5, 2012.
	
	\bibitem{ReyBelletSpiliopoulos2015}
	Luc Rey-Bellet and Konstantinos Spiliopoulos.
	\newblock {Irreversible Langevin samplers and variance reduction: a large
		deviations approach}.
	\newblock {\em Nonlinearity}, 28(7):2081--2103, 2015.
	
	\bibitem{roberts1996quantitative}
	Gareth Roberts, Jeffrey Rosenthal, and Others.
	\newblock {Quantitative bounds for convergence rates of continuous time Markov
		processes}.
	\newblock {\em Electronic Journal of Probability}, 1, 1996.
	
	\bibitem{Sherlock2017}
	Chris Sherlock and Alexandre~H Thiery.
	\newblock {A Discrete Bouncy Particle Sampler}, 2017.
	
	\bibitem{Stramer1999}
	O.~Stramer and R.~L. Tweedie.
	\newblock {Langevin-Type Models I: Diffusions with Given Stationary
		Distributions and their Discretizations}.
	\newblock {\em Methodology and Computing in Applied Probability}, 306:283--306,
	1999.
	
	\bibitem{TuritsynChertkovVucelja2011}
	Konstantin~S. Turitsyn, Michael Chertkov, and Marija Vucelja.
	\newblock {Irreversible Monte Carlo algorithms for efficient sampling}.
	\newblock {\em Physica D: Nonlinear Phenomena}, 240(4-5):410--414, feb 2011.
	
	\bibitem{Twe94}
	R.~L. Tweedie.
	\newblock Topological conditions enabling use of {H}arris methods in discrete
	and continuous time.
	\newblock {\em Acta Appl. Math.}, 34(1-2):175--188, 1994.
	
	\bibitem{Vanetti2017}
	Paul Vanetti, Alexandre Bouchard-C{\^{o}}t{\'{e}}, George Deligiannidis, and
	Arnaud Doucet.
	\newblock {Piecewise Deterministic Markov Chain Monte Carlo}.
	\newblock {\em arXiv preprint arXiv: 1707.05296}, jul 2017.
	
	\bibitem{Wu2017}
	Changye Wu and Christian~P. Robert.
	\newblock {Generalized Bouncy Particle Sampler}.
	\newblock {\em arXiv preprint arXiv: 1706.04781}, jun 2017.
	
\end{thebibliography}

 \end{document}